\documentclass[11pt, letterpaper]{amsart}
\usepackage[left=1in,right=1in,bottom=1in,top=1in]{geometry}

\usepackage{amsmath,amsthm,amssymb}
\usepackage{mathrsfs}
\usepackage[all,cmtip]{xy}
\usepackage{cite}
\usepackage{graphicx}
\usepackage{hyperref}
\usepackage[shortlabels]{enumitem}
\usepackage{multicol}

\newcommand{\ZZ}{\mathbf Z}
\newcommand{\QQ}{\mathbf Q}

\newcommand{\kbar}{\overline{k}}

\newcommand{\FF}{\mathbf F}

\newcommand{\gl}{\mathfrak{gl}}

\newcommand{\ZZhat}{\widehat{\ZZ}}

\newcommand{\g}{\mathfrak{g}}

\newcommand{\m} {\mathfrak m}

\DeclareFontFamily{OT1}{rsfs}{}
\DeclareFontShape{OT1}{rsfs}{n}{it}{<-> rsfs10}{}
\DeclareMathAlphabet{\mathscr}{OT1}{rsfs}{n}{it}

\renewcommand{\O}{\mathcal{O}}

\newcommand{\on}[1]{\operatorname{#1}}
\newcommand{\Hom}{\on{Hom}}

\newcommand \tensor[1] {\otimes_{#1}}

\newcommand{\Aut}{\on{Aut}}

\newcommand{\End}{\on{End}}

\newcommand{\Spec}{\on{Spec}}
\newcommand{\Lie}{\on{Lie}}

\newcommand{\GL}{\on{GL}}

\newcommand{\SO}{\on{SO}}

\newcommand{\Mat}{\on{Mat}}

\newcommand{\into}{\hookrightarrow}

\theoremstyle{plain}
\newtheorem{lem}{Lemma}
\newtheorem{thm}[lem]{Theorem}
\newtheorem{prop}[lem]{Proposition}
\newtheorem{cor}[lem]{Corollary}

\newtheorem{fact}[lem]{Fact}

\theoremstyle{definition}
\newtheorem{defn}[lem]{Definition}
\newtheorem{example}[lem]{Example}
\newtheorem{remark}[lem]{Remark}

\newcommand{\ttm}[4]{\begin{pmatrix}
#1 & #2 \\
#3 & #4
\end{pmatrix}}

\numberwithin{equation}{section}
\numberwithin{lem}{section}

\newcommand{\ad}{\on{ad}}

\newcommand{\Sp}{\on{Sp}}
\newcommand{\Orth}{\on{O}}
\newcommand{\GSp}{\on{GSp}}
\newcommand{\GO}{\on{GO}}

\newcommand{\Gm}{\mathbf{G}_m}

\newcommand{\GAut}{\on{GAut}}

\newcommand{\Fil}{\on{Fil}}

\newcommand{\Nil}{\on{Nil}}

\newcommand{\rhobar}{{\overline{\rho}}}
\newcommand{\adrho}{\on{ad}(\rhobar)}
\newcommand{\adzerorho}{\on{ad}^0(\rhobar)}

\newcommand{\Mbar}{\overline{M}}
\newcommand{\Nbar}{\overline{N}}
\newcommand{\Phibar}{{\overline{\Phi}}}

\newcommand{\onto}{\twoheadrightarrow}
\newcommand{\inverselimit}{\varprojlim}

\newcommand{\ob}{\on{ob}}
\newcommand{\C}{\mathcal{C}}
\newcommand{\Chat}{\widehat{\C}}
\newcommand{\D}{\mathcal{D}}

\newcommand{\Spf}{\on{Spf}}

\newcommand{\tildetau}{{\widetilde{\tau}}}
\newcommand{\Res}{\on{Res}}
\newcommand{\Ind}{\on{Ind}}

\newcommand{\id}{\on{id}}

\newcommand{\Ad}{\on{Ad}}

\newcommand{\zfrak}{\mathfrak{z}}

\newcommand{\Span}{\on{Span}}

\newcommand{\Ghat}{\widehat{G}}

\newcommand{\blockmatrix}{\ttm}

\newcommand{\nr}{{\on{nr}}}
\newcommand{\tame}{{\on{t}}}

\newcommand{\Sets}{{\on{Sets}}}
\newcommand{\mr}{\on{m.r.}}

\begin{document}
\title{Minimally Ramified Deformations when $\ell \neq p$}
\author{Jeremy Booher}
\email{jeremybooher@math.arizona.edu}
\date{\today}
\address{Department of Mathematics\\
  University of Arizona\\
  Tucson, Arizona 85721}
  

\begin{abstract}
Let $p$ and $\ell$ be distinct primes, and $\rhobar$ be an orthogonal or symplectic representation of the absolute Galois group of an $\ell$-adic field over a finite field of characteristic $p$.  We define and study a liftable deformation condition of lifts of $\rhobar$ ``ramified no worse than $\rhobar$'', generalizing the minimally ramified deformation condition for $\GL_n$ studied in \cite{cht08}.  The key insight is to restrict to deformations where an associated unipotent element does not change type when deforming.  This requires an understanding of nilpotent orbits and centralizers of nilpotent elements in the relative situation, not just over fields.
\end{abstract}

\maketitle

\tableofcontents

\section{Introduction}

Let $\ell$ and $p$ be primes, $L$ be a finite extension of $\QQ_\ell$, and $\Gamma_L$ be the absolute Galois group of $L$.  Suppose $\O$ is the ring of integers in a $p$-adic field with residue field $k$.  For a reductive group $G$, it is important to study deformations of a continuous representation $\rhobar : \Gamma_L \to G(k)$.  Information about the universal deformation ring, and quotients corresponding to restricted classes of deformations, have many applications, for example to producing congruences between modular forms, proving modularity lifting theorems, and understanding generalizations of Serre's conjecture and of the Breuil-M\'{e}zard conjecture.  The case $G = \GL_n$ has received the most attention.  In this paper, we assume $\ell \neq p$ and generalize the minimally ramified deformation condition for $\GL_n$ studied by Clozel, Harris, and Taylor \cite[\S 2.4.4]{cht08} to symplectic and orthogonal groups.

This question was originally motivated by the problem of producing geometric deformations of representations of the absolute Galois group of a number field using a generalization of a method introduced by Ramakrishna ~\cite{ram99} \cite{ram02}.  For use in Ramakrishna's method, we would like to define a deformation condition of lifts which are ``ramified no worse than $\rhobar$,'' such that the resulting deformation condition is liftable despite the fact that the unrestricted deformation condition for $\rhobar$ may not be liftable.  When $G = \GL_n$, the minimally ramified deformation condition defined in \cite[\S 2.4.4]{cht08} works.  Attempting to generalize the argument of \cite[\S 2.4.4]{cht08} to groups besides $\GL_n$ leads to a deformation condition based on parabolics which is \emph{not} liftable.  Instead, inspired by the arguments of 
\cite[\S3]{taylor08} we define a deformation condition for symplectic and orthogonal groups based on deformations of a nilpotent element of $\Lie G_k$.
This condition is liftable, which illustrates how genuinely new ideas are needed to study the deformation rings for representations valued in groups besides $\GL_n$. 

In \S\ref{sec:mrd}, we define a \emph{minimally ramified deformation condition} for symplectic and orthogonal groups after extending the residue field $k$.  This extension is harmless for the original application, and for that application it is also convenient to consider deformations with a fixed similitude character.  Our main result is the following, which is precisely the local input needed in Ramakrishna's method in the $\ell \neq p$ case.

\begin{thm} \label{thm:maintheorem}
Let $G$ be $\GSp_n$ or $\GO_n$ over $\ZZ_p$ with $p>n$, and let $\rhobar: \Gamma_L \to G(k)$ be a continuous representation with $\ell \neq p$.  After extending $k$, the minimally ramified deformation condition with fixed similitude character is a liftable deformation condition (in the sense of Definition~\ref{defn:liftable}), and its tangent space has dimension $\dim H^0(\Gamma_L ,\adzerorho)$.
\end{thm}

This can equivalently be expressed as exhibiting a formally smooth quotient of the universal lifting ring $R^\square_\rhobar$.  In this paper, we study only the local theory: the applications to producing geometric lifts are discussed in \cite{boohergeometric}.  
In the remainder of the introduction, we will sketch how to correctly generalize the minimally ramified deformation condition introduced for $\GL_n$ and analyze it.  The strategy could work for general $G$, but several pieces of the argument are specific to orthogonal or symplectic groups (or $\GL_n$), which was all that was needed for the original application.

The first step in \cite[\S 2.4.4]{cht08} is to reduce to studying certain tamely ramified representations.   They reduce the problem to defining a nice class of deformations for representations of the group $T_q :=  \ZZhat \ltimes \ZZ_p$, where the first factor is generated by a Frobenius $\phi$ and the second by an element $\tau$ in the inertia group.  They satisfy the relation $\phi \tau \phi^{-1} = q \tau$ for some $q$ prime to $p$.  This reduction generalizes without surprises to symplectic and orthogonal groups in \S\ref{sec:mr} (but the argument is genuinely restricted to orthogonal and symplectic groups as it relies heavily on the pairing).

The second step is to specify when a lift of $\rhobar : T_q \to \GL_n(k)$ is ``ramified no worse than $\rhobar$''.  For a coefficient ring $R$, a deformation $\rho : T_q \to \GL_n(R)$ is \emph{minimally ramified} according to \cite{cht08} when the natural $k$-linear map
 \begin{align} \label{eq:mr}
  \ker \left( (\rho(\tau) - 1_n)^i \right) \tensor{R} k \to \ker\left( (\rhobar(\tau)-1_n)^i \right)
 \end{align}
is an isomorphism for all $i$.  The deformation condition is analyzed as follows:
\begin{itemize}
 \item  defining $V_i = \ker\left( (\rhobar(\tau)-1_n)^i \right)$ gives a flag
 \[
  0 \subset V_r \subset V_{r-1} \subset \ldots \subset V_1 \subset k^n.
 \]
This flag determines a parabolic $k$-subgroup $\overline{P} \subset \GL_n$ (points which preserve the flag) such that $\rhobar(\tau) \in (\mathcal{R}_u \overline{P})(k)$ and $\rhobar(\phi) \in \overline{P}(k)$;

\item  lift $\overline{P}$ to a parabolic subgroup $P$ of $\GL_n$.  The deformation functor of such lifts is formally smooth, and for any minimally ramified deformation $\rho$ over $R$ there is a choice of such $P$ for which $\rho(\tau) \in (\mathcal{R}_u P)(R)$ and $\rho(\phi) \in P(R)$.  Conversely, any $\rho$ with this property is minimally ramified;

\item  Finally, for the standard block-upper-triangular choice of $P$, one shows the deformation functor
\[
 \{ (T,\Phi) : T \in \mathcal{R}_u P, \Phi \in P , \Phi T \Phi^{-1} = T^q, \overline{T} = \rhobar(\tau) , \overline{\Phi} = \rhobar(\phi) \}
\]
is formally smooth by building the universal lift over a power series ring: this uses explicit calculations with block-upper-triangular matrices.
\end{itemize}

To generalize beyond $\GL_n$, we need to replace \eqref{eq:mr} with a more group-theoretic criterion.  The naive generalization is to associate a parabolic $\overline{P}$ to $\rhobar$ and then use the following definition.

\begin{defn} \label{defn:ramifiedwrtp}
For a coefficient ring $R$, say a lift $\rho : T_q \to G(R)$ is \emph{ramified with respect to $\overline{P}$} provided that there exists a parabolic $R$-subgroup $P \subset G_R$ lifting $\overline{P}$ such that $\rho(\tau) \in (\mathcal{R}_u P )(R)$ and $\rho(\phi) \in P(R)$.
\end{defn}

This idea does not work.  Let us focus on the symplectic case to illustrate what goes wrong.

The first problem is to associate a parabolic subgroup to $\rhobar$.  Recall that parabolic subgroups of a symplectic group correspond to isotropic flags $0 \subset V_1 \subset \ldots \subset V_r \subset V_r^\perp \subset \ldots \subset V_1^\perp \subset k^{2n}$.   There is no reason that the flag determined by \eqref{eq:mr} is isotropic, so we would need some other method of producing a parabolic $\overline{P}$ such that $\rhobar(\tau) \in (\mathcal{R}_u \overline{P})(k)$.  In \cite{boreltits71}, Borel and Tits give a natural way to associate to the unipotent $\rhobar(\tau)$ a smooth connected unipotent $k$-subgroup of $G$.  The normalizer of this subgroup is always parabolic and so gives a candidate for $\overline{P}$.  However, working out examples in $\GL_n$ for small $n$ shows that this produces a different parabolic than the one determined by \eqref{eq:mr}.  This raises the natural question of how sensitive the smoothness of the deformation condition is to the choice of parabolic.

This leads to the second, larger problem: there are examples such that for \emph{every} parabolic $\overline{P}$ satisfying $\rhobar(\tau) \in (\mathcal{R}_u \overline{P})(k)$, not all deformations ramified with respect to $\overline{P}$ are liftable.
\begin{example} \label{ex:unliftable}
Take $L = \QQ_{29}$ and $k = \FF_7$.  
Consider the representation $\rhobar : T_{29} \simeq \ZZhat \ltimes \ZZ_7 \to \GSp_4(\FF_7)$ defined by
\[
 \rhobar(\tau ) = \begin{pmatrix}
                  1 & 0 & 1 & 0 \\
                  0 & 1 & 0 & 0\\
                  0 & 0 & 1 & 0 \\
                  0 & 0 & 0 & 1
                  \end{pmatrix}
\quad \text{and} \quad
\rhobar(\phi) = \begin{pmatrix}
                 1 & -1 & 0 & 0 \\
                 0 & 1 & 0 & 0\\
                 0 & 0 & 1 & 0 \\
                 0 & 0 & 1 & 1
                \end{pmatrix}.
\]
The deformation condition of lifts ramified relative to a parabolic $\overline{P}$ of $\GSp_4$ whose unipotent radical contains $\rhobar(\tau)$ is not liftable for any choice of $\overline{P}$: there are lifts to the dual numbers that do not lift to $\FF_7[\epsilon]/(\epsilon^3)$.  This is easy to check with a computer algebra system such as \cite{sage}, since the existence of lifts can be reduced to a problem in linear algebra.  This is a general phenomenon, which we will explain conceptually in \S\ref{sec:parabolicdefcondition}.
\end{example}

The correct approach is to define a lift $\rho : T_q \to G(R)$ to be minimally ramified if $\rho(\tau)$ has ``the same unipotent structure'' as $\rhobar(\tau)$.  It is more convenient to work with nilpotent elements, using the exponential and logarithm maps (defined for nilpotent and unipotent elements since $p > n$).  There are combinatorial parametrizations of nilpotent orbits of algebraic groups over an algebraically closed field, for example in terms of partitions or root data, which make precise the notion that the values of $N \in \g_\O$ in the special and generic fiber lie in the same nilpotent orbit.  In particular, for each nilpotent orbit $\sigma$, we  use the results of \S\ref{sec:integralrep} to choose particular elements $N_\sigma \in \g_\O$ with this property lifting $\Nbar \in \g_k$.  In \S\ref{sec:purenilplift}, we define the \emph{pure nilpotents lifting $\Nbar$} to be the $\widehat{G}(R)$-conjugates of $N_\sigma$ for a coefficient ring $R$.

\begin{example} \label{ex:changeorbit}
For example, let $G = \GL_3$ and 
\[
 \Nbar = \begin{pmatrix}
        0 & 1 & 0 \\
        0 & 0 & 0\\
        0 & 0 & 0
       \end{pmatrix}
\]
Consider the lifts
\[
 N_1 = \begin{pmatrix}
      0 & 1 & 0\\
      0 & 0 & 0\\
      0 & 0 & 0
     \end{pmatrix} \in \g \quad \text{and} \quad     
N_2 = \begin{pmatrix}
       0 & 1 & 0\\
       0 & 0 & p \\
       0 & 0 & 0
      \end{pmatrix} \in \g.
\]
Both are nilpotent under the embedding of $\O$ into its fraction field $K$.  The images of $N_1$ in $\g_K$ and $\g_k$ both lie in the nilpotent orbit corresponding to the partition $2+1$, so $N_1$ is an example of the type of nilpotent lift we want to consider.  On the other hand, the image of $N_2$ in $\g_K$ lies in the nilpotent orbit corresponding to the partition $3$, while the image on $\g_k$ lies in the orbit corresponding to $2+1$, so we do not want to use it.  The pure nilpotents lifting $\Nbar$ are $\widehat{G}(R)$-conjugates of $N_1$.
\end{example}

We finally define a lift
$\rho : T_q \to G(R)$ to be \emph{minimally ramified} provided $\rho(\tau)$ is the exponential of a pure nilpotent lifting $\log \rhobar(\tau) = \Nbar$.  Proposition~\ref{prop:smoothnessmrtame} shows that this deformation condition is liftable.
The main technical fact needed to analyze this deformation condition is that the scheme-theoretic centralizer $Z_G(N_\sigma)$ is smooth over $\O$ for $N_\sigma$ as above.  The smoothness of such centralizers over algebraically closed fields is well-understood, and in \S\ref{sec:centralizers} we study $Z_G(N_\sigma)$ and show that $Z_G(N_\sigma)$ is \emph{flat} over $\O$ and hence smooth.  Lemma~\ref{lem:flatnesscriterion} gives a criterion for flatness that is easy to verify for classical groups, which suffices for our applications.  We can reduce checking $\O$-flatness to the problem of finding elements $g \in Z_G(N_\sigma)(\O)$ such that $g_k$ lies in any specified component of $Z_{G_k}(\Nbar) / Z_{G_k}(\Nbar)^\circ$.  There are difficulties beyond the classical cases due to the varied structure of $\pi_0( Z_G(\Nbar)_{\kbar})$ in general.  

\begin{remark}
It is a fortuitous coincidence (for \cite{cht08}) that for $\GL_n$ the lifts minimally ramified in the preceding sense are exactly the lifts ramified with respect to a parabolic subgroup of $G$.  This rests on the fact that all nilpotent orbits of $\GL_n$ are Richardson orbits (see \S\ref{sec:parabolicdefcondition} for details).
\end{remark}

\subsection{Structure of the Paper}
Section \ref{sec:deformations} discusses deformation conditions, deformation rings, and lifting rings.  Section \ref{sec:nilprep} constructs integral representatives for nilpotent orbits, and defines the notion of a pure nilpotent lift.  This notion requires a study of the $\O$-smoothness of $Z_G(N)$, which is carried out in \S\ref{sec:centralizers}.  Finally Sections \ref{sec:mrtame} and \ref{sec:mr} define and study the minimally ramified deformation condition, first in a special tamely ramified case and then in general.

\subsection{Notation and Assumptions} \label{sec:notation}

Throughout the paper, $\ell$ and $p$ will be distinct primes, and $\O$ will be a discrete valuation ring with residue field $k$ of characteristic $p$.  We will ultimately work with orthogonal or symplectic (similitude) groups, or $\GL_n$ over $\O$, although the strategy of the argument (but not the details) would work in greater generality.  Since reductive group schemes have connected fibers (a restriction going back to \cite[XIX, 2.7]{sga3} to avoid the component group jumping across fibers), and since $\GO_m$ may be disconnected, the natural class of group schemes to work with are what we call \emph{almost-reductive groups}.  
 By this we mean a smooth separated group scheme over $\O$ such that the identity components of the fibers are reductive.  Then $G^\circ$ is a reductive $\O$-subgroup scheme of $G$ and $G/G^\circ$ is a separated \'{e}tale $\O$-group scheme of finite presentation~\cite[Proposition 3.1.3 and Theorem 5.3.5]{conrad14}.   Furthermore, by a result of Raynaud $G$ is affine as it is a flat, separated, and of finite type with affine generic fiber over the discrete valuation ring $\O$ \cite[Proposition 3.1]{py06}.  

We will often assume that $p$ is very good for $G$: in cases of interest this means that $p\neq 2$ if $G$ is orthogonal or symplectic, and $p \nmid n$ when $G = \GL_n$.  To make uniform statements, we say that characteristic zero is very good for any $G$. 

We will also work with a nilpotent $\Nbar \in \g_k$ associated to a continuous representation $\rhobar : \Gamma_L \to G(k)$, where $L$ is an $\ell$-adic field.  We will eventually impose additional hypotheses, including:
\begin{enumerate}[label={(A\arabic*)}]
 \item \label{assumption1}$G$ is $\GSp_n$, $\GO_n$ or $\GL_n$ and $p$ is very good for $G$;
 \item $p> n$;
 \item $k$ and $\O$ are large enough so that $q$, the size of the residue field of $L$, is a square in $\O^\times$, and $\O^\times$ includes square roots of $-1$ and $2$;
 \item \label{assumption4} $k$ and $\O$ are large enough so there exists a pure nilpotent $N_\sigma \in \g$ lifting $\Nbar$ for which $Z_G(N_\sigma)$ is smooth. 
 \end{enumerate}

\subsection{Acknowledgements}
This work forms part of my thesis \cite{booher16}, and 
I am extremely grateful for the generosity and support of my advisor Brian Conrad, and for his extensive and helpful comments on drafts of my thesis.  The thesis was originally submitted as a single paper before being split into this paper and \cite{boohergeometric}: I thank the original referee for a careful reading.  

\section{Deformations of Galois Representations} \label{sec:deformations}

We recall some facts about the deformation theory for Galois representations: a basic reference is \cite{mazur95}, with the extension to algebraic groups beyond $\GL_n$ discussed in \cite{tilouine96}.  While we are mainly concerned with classical groups, there are no problems with doing so for any smooth group scheme $G$ over a discrete valuation ring $\O$ with residue field $k$ of characteristic $p$.  

Let $\Gamma$ be a pro-finite group satisfying the following finiteness property: for every open subgroup $\Gamma_0 \subset \Gamma$, there are only finitely many continuous homomorphisms from $\Gamma_0$ to $\ZZ/ p \ZZ$.  This is true for the absolute Galois group of a local field and for the Galois group of the maximal extension of a number field unramified outside a finite set of places.

Let $\Chat_\O$ be the category of coefficient $\O$-algebras: complete local Noetherian rings with residue field $k$, with morphisms local homomorphisms inducing the identity map on $k$ and with the structure morphism a map of coefficient rings.  Let $\C_\O$ denote the full subcategory of Artinian coefficient $\O$-algebras.  Recall that a \emph{small} surjection of coefficient $\O$-algebras $f : A_1 \to A_0$ is a surjection such that $\ker(f) \cdot \m_{A_1} = 0$.

For $A \in \Chat_\O$, define 
\[
 \Ghat(A) := \ker( G(A) \to G(k))
\]
We are interested in deforming a fixed $\rhobar : \Gamma \to G(k)$.  Let $\g = \Lie G$.

\begin{itemize}
 \item   Let $f: A_1 \to A_0$ be a morphism in $\Chat_{\O}$ and $\rho_0 : \Gamma \to G(A_0)$ a continuous homomorphism.  A \emph{lift} of $\rho_0$ to $A_1$ is a continuous homomorphism $\rho_1 : \Gamma \to G(A_1)$ such that the following diagram commutes:
\[
\xymatrix{
\Gamma \ar[r]^{\rho_1} \ar[rd]^{\rho_0} & G(A_1) \ar[d]^f\\
& G(A_0)
} 
\]
Define the functor $D_{\rhobar,\O}^{\square}: \Chat_\O \to \Sets$ by sending a coefficient $\O$-algebra $A$ to the set of lifts of $\rhobar$ to $A$.  

\item  With the notation above, two lifts $\rho$ and $\rho'$ of $\rhobar$ to $A_1 \in \C_\O$ are \emph{strictly equivalent} if they are conjugate by an element of $\Ghat(A_1)$.  A \emph{deformation} of $\rho_0$ to $A_1$ is a strict equivalence class of lifts.  Define the functor $D_{\rhobar,\O}: \Chat_\O \to \Sets$ by sending a coefficient $\O$-algebra $A$ to the set of deformations of $\rhobar$ to $A$.  
\end{itemize}

We will drop the subscript $\O$ when it is clear from context.

\begin{fact} \label{fact:representable}
The functor $D^\square_{\rhobar,\O}$ is representable.  When $\g_k^\Gamma = \Lie(Z_G)_k$, the functor $D_{\rhobar, \O}$ is representable.  
\end{fact}

The first part is simple, the second is a reformulation of \cite[Theorem 3.3]{tilouine96}.

The representing objects are denoted $R^\square_{\rhobar, \O}$ and (when it exists) $R_{\rhobar, \O}$.  The former is called the universal lifting ring, while the latter is the universal deformation ring.  While we usually care about deformations, it is technically easier to work with lifts.

This deformation theory is controlled by Galois cohomology.  Let $\adrho$ denote the representation of $\Gamma$ on $\g_k$ via the adjoint representation.  Letting $G'$ be the derived subgroup of $G^\circ$ with Lie algebra $\g'$, we also consider the representation $\adzerorho$ of $\Gamma$ on $\g'_k$.  As $p$ is very good, we have $\g_k = \g'_k \oplus \zfrak_\g$ where $\zfrak_\g$ is the Lie algebra of $Z_G$.  The condition in Fact~\ref{fact:representable} is just that $H^0(\Gamma,\adrho) = \zfrak_\g$, or equivalently that  $H^0(\Gamma,\adzerorho)=0$.  In general, since $p$ is very good the natural map $H^i(\Gamma,\adzerorho) \to H^i(\Gamma,\adrho)$ is injective for all $i$; we often use this without comment.

We can use the first order exponential map \cite[\S 3.5]{tilouine96} to understand the tangent space.  Recall that for a smooth $\O$-group scheme $G$, and a small surjection $f : A \to A/I$ of coefficient rings ($I \cdot \m_{A} = 0$), smoothness gives an isomorphism
\[
  \exp: \g \tensor{k} I \simeq \ker( G(A) \to G(A/I)) = \ker( \Ghat(A) \to \Ghat(A/I)).
\]
The tangent space $D_{\rhobar,\O}(k[\epsilon]/\epsilon^2)$ is identified with $H^1(\Gamma, \adrho)$:  Under this isomorphism, the cohomology class of a $1$-cocycle $\tau$ corresponds to the lift $\rho(g) =  \exp(\epsilon \tau(g)) \rhobar(g)$.
For the universal lifting ring $R^\square_{\rhobar,\O}$, the tangent space is identified with the $k$-vector space $Z^1(\Gamma,\adrho)$ of (continuous) $1$-cocycles of $\Gamma$ valued in $\adrho$.

\begin{remark} \label{rem:framedvdef}
   We also observe that
 \[
  \dim_k Z^1(\Gamma,\adrho) - \dim_k H^1(\Gamma, \adrho) = \dim_k B^1(\Gamma,\adrho) = \dim_k \g - \dim_k H^0(\Gamma,\adrho)
 \]
since the space of coboundaries admits a surjection from $\adrho$ with kernel $\adrho^\Gamma$.  This will be useful when comparing dimensions of lifting rings and deformation rings that are smooth.
\end{remark}

We will want to studying special classes of deformations. 

\begin{defn} \label{defn:deformationcondition}
A \emph{lifting condition} is a sub-functor $\D^\square \subset D^\square_{\rhobar, \O} : \C_\O \to \Sets$ such that:
\begin{enumerate}
 \item  For any coefficient ring $A$, $\D^\square(A)$ is closed under strict equivalence.
  \item  \label{defcon2} Given a Cartesian diagram in $\C_\O$
 \[
  \xymatrix{
  A_1 \times_{A_0} A_2 \ar[r]^-{\pi_2} \ar[d]^-{\pi_1} & A_2\ar[d]\\
  A_1 \ar[r] & A_0
  }
 \]
 and $\rho \in D^\square_{\rhobar, \O}(A_1 \times_{A_0} A_2)$, we have $\rho \in \D^\square(A_1 \times_{A_0} A_2)$ if and only if $\D^\square(\pi_1) \circ \rho \in \D^\square(A_1)$ and $\D^\square(\pi_2) \circ \rho \in \D^\square(A_2)$.  
\end{enumerate}
As it is closed under strict equivalence, we naturally obtain a \emph{deformation condition}, a sub-functor $\D \subset D_{\rhobar,\O}$.
\end{defn}

By Schlessinger's criterion \cite[Theorem 2.11]{schlessinger68} being a lifting condition is equivalent to the functor $\D^\square$ being pro-representable.  Likewise, the deformation condition $\D$ associated to a lifting condition $\D^\square$ is pro-representable provided that $D_{\rhobar, \O}$ is.

The tangent space of a deformation condition $\D$ is a $k$-subspace of $H^1(\Gamma,\adrho)$, and will be denoted by $H^1_{\D}(\Gamma,\adrho)$.  For a small surjection $A_1 \to A_0$ and $\rho \in \D(A_0)$, the set of deformations of $\rho$ to $A_1$ subject to $\D$ is a 
$H^1_{\D}(\Gamma,\adrho)$-torsor.  This torsor-structure is compatible with the action of the unrestricted tangent space to $D_\rhobar$ on the space of all deformations of $\rho$ to $A_1$.

\begin{example} \label{ex:liftingcondition}
Suppose $G$ is almost-reductive, and let $G'$ be the derived group of $G^\circ$.
 The most basic examples of deformation conditions are the conditions imposed by fixing the lift of the homomorphism $\Gamma \to (G/G')(k)$.  To be precise, for the quotient map $\mu : G \to G/G' =: S$, a fixed $\nu : \Gamma \to S(\O)$ lifting $\mu \circ \rhobar$, and $A \in \Chat_\O$ with structure morphism $\imath : \O \to A$, we define a deformation condition $\D_\nu \subset \D_\rhobar$ by
 \[
  \D_\nu(A) = \{ \rho \in \D_\rhobar(A) |  \Gamma \to G(A) : \mu_A \circ \rho = \imath \circ \nu_A \}.
 \]
One checks this is a deformation condition.  Its tangent space is $H^1(\Gamma,\adzerorho)$ since $p$ is very good.  We define $\D_\nu^\square$ similarly.  

Another important example is the \emph{unramified} deformation condition for a non-archimedean place $v$ where $\rhobar$ is unramified: this consists of lifts that are unramified (possibly with a specified choice of $\nu$).  The tangent space is $H^1_{\nr}(\Gamma_v,\adrho)$ (respectively $H^1_{\nr}(\Gamma_v,\adzerorho)$).
\end{example}

\begin{defn} \label{defn:liftable}  A deformation condition $\D$ is \emph{liftable} (over $\O$) if for all small surjections $f : A_1 \to A_0$ of coefficient $\O$-algebras the natural map
\[
 \D(f) : \D(A_1) \to \D(A_0)
\]
is surjective.
\end{defn}

A geometric way to check local liftability is to show that the corresponding deformation ring (when it exists) is smooth.  Obviously it suffices to check liftability for lifts instead of deformations, so we can work with the lifting deformation ring and avoid representability issues for $\D_\rhobar$.

\begin{example} \label{ex:liftability}
The unramified deformation condition is liftable: an unramified lift is completely determined by the image of Frobenius in $G(A_0)$, and $G$ is smooth over $\O$.

When attempting to lift with a fixed lift $\nu$ of $\Gamma \to (G/G')(k)$, the obstruction to lifting is measured by a $2$-cocycle $\ob(\rho_0)$ that lies in $H^2(\Gamma,\adzerorho)$.  To see this, recall that the obstruction cocycle is defined by picking a continuous set theoretic lift $\rho_1$ of a given $\rho_0 : \Gamma_K \to G(A_0)$: the $2$-cocycle records the failure of $\rho_1$ to be a homomorphism.  By choosing the lift $\Gamma_K \to G(A_1)$ so that $\Gamma_K \to (G/G')(A_0)$ agrees with $\nu$ (as we may easily do since $\ker \rho_0$ is open in $\Gamma_K$), the obstruction cocycle takes values in $\adzerorho$.
\end{example}

\section{Representatives for Nilpotent Orbits and Pure Nilpotents} \label{sec:nilprep}

As a first step on the road to defining the minimally ramified deformation condition, we study integral representatives of nilpotent orbits and then define pure nilpotent lifts.  Useful background about nilpotent orbits is collected in \cite{jantzen04}. 
 We focus on classical groups, so consider $G = \GL_n$, $G= \Sp_m$, or $G = \Orth_m$ over a discrete valuation ring $\O$ with residue field $k$ of characteristic $p>0$.  Assume $p$ is very good for $G_k$.  Let $\g = \Lie G$ and $K$ be the field of fractions of $\O$.
 
\subsection{Integral Representatives} \label{sec:integralrep} The nilpotent orbits for $G$ over an algebraically closed field of good characteristic can be classified by combinatorial data $\C$ that is independent of the characteristic.  For classical groups, nilpotent orbits can be classified by their Jordan canonical form in terms of partitions .  For a partition $\sigma \in \C$, let $\O_{F,\sigma} \subset \g_F$ denote the corresponding orbit over the algebraically closed field $F$.  
 For $\sigma \in \C$, we seek elements
\begin{equation} \label{eq:ncondition}
 N_\sigma \in \g \,\text{ such that } \, (N_\sigma)_k \in O_{\kbar,\sigma} \, \text{ and } \, (N_\sigma)_K \in O_{\overline{K},\sigma}.  
\end{equation}
This makes precise the statement that $N_K$ and $N_k$ ``lie in the same nilpotent orbit.''

\begin{remark}\label{rmk:rootreps}
For a general reductive group scheme $G$, the Bala-Carter classification can be interpreted as giving a characteristic-free classification of nilpotent orbits, allowing a generalization of the condition in \eqref{eq:ncondition}.  One can obtain such $N_\sigma$ in terms of root data following \cite[III.4.29]{ss70}.  We need the additional information provided by the concrete description in the symplectic and orthogonal cases to analyze the centralizer $Z_G(N)$ as an $\O$-scheme, so do not use this.  
\end{remark}


\begin{example} \label{ex:glnrep}
Nilpotent orbits for $\GL_n$ correspond to partitions $n = n_1 + n_2 + \ldots + n_r$.  For a partition $\sigma$ of $n$, Let $N_\sigma \in \g$ be the nilpotent matrix in Jordan canonical form whose blocks (in order) are of sizes $n_1, n_2, \ldots , n_r$.  Clearly $N_\sigma$ has entries in $\O$ and satisfies \eqref{eq:ncondition}.
\end{example}

For symplectic and orthogonal groups, we can produce the desired $N_\sigma$ using a minor extension of the classical results known over algebraically closed fields \cite[\S1]{jantzen04}.
Let $G = \Sp_{m}$ with $m = 2n$, or $G=\Orth_m$ with $m = 2n$ or $m = 2n+1$.  We assume $n \geq 2$.
Recall that $\Sp_{m}$ and $\Orth_m$ are defined using standard pairings on a free $\O$-module $M$ of rank $m$.  For $m = 2n$, the \emph{standard alternating pairing} $\varphi_{\on{std}}$ on $\O^m$ is the one given by the block matrix
\[
 \blockmatrix{0}{I'_{n}}{-I'_{n}}{0},
\]
where $I'_{n}$ denotes the anti-diagonal matrix with $1$'s on the diagonal.
The \emph{standard symmetric pairing} $\varphi_{\on{std}}$ on $\O^m$ is the one given by the matrix $I'_m$.

\begin{remark} \label{rmk:somorbits}
 We chose to work with $\Orth_m$ instead of $\SO_m$, as the classification is cleaner for $\Orth_m$.  The nilpotent orbits are almost the same for $\SO_m$, except that certain nilpotent orbits of $\Orth_m$ (the ones where the partition contains only even parts) split into two $\SO_m$-orbits \cite[Proposition 1.12]{jantzen04} (conjugation by an element of $\Orth_m$ with determinant $-1$ carries one such orbit into the other).
\end{remark}

\begin{defn}
Let $\sigma$ denote a partition $m = m_1 + m_2 + \ldots +m_r$ of $m$.  It is \emph{admissible} if
\begin{itemize}
 \item every even $m_i$ appears an even number of times when $G =\Orth_m$ ;
 \item every odd $m_i$ appears an even number of times when $G = \Sp_m$.
\end{itemize}
\end{defn}

The admissible partitions of $m$ are in bijection with nilpotent orbits of $\Sp_m$ or $\Orth_m$  over any algebraically closed field of good characteristic \cite[Theorem 1.6]{jantzen04}.  The corresponding orbit is the intersection of $\g \subset \gl_m$ with the $\GL_m$-orbit corresponding to that partition of $m$.  Note that $\GL_m$-orbit representatives in Jordan canonical form need not lie in $\g$.  

We will construct nilpotents together with a pairing, and then show how to relate the constructed pairing to the standard pairings used to define $G$.  Let $\epsilon = 1$ in the case of $\Orth_m$, and $\epsilon = -1$ in the case of $\Sp_m$.

\begin{defn} \label{defn:nilpotentpieces}
Let $d\geq 2$ be an integer.  Define $M(d) = \O^d$, with basis $v_1,\ldots v_d$ and a perfect pairing $\varphi_d$ such that
\[
 \varphi_d(v_i,v_j) = \begin{cases}
                     (-1)^i, & i+j = d+1\\
                     0, & \text{otherwise}
                    \end{cases}
\]
(alternating for even $d$, symmetric for odd $d$).
Define a nilpotent $N_d \in \End(M(d))$ by $N_d v_i = v_{i-1}$ for $1 < i \leq d$ and $N_d v_1 =0$.

Similarly, define $M(d,d) = \O^{2d}$ with basis $v_1,\ldots v_d, v_1',\ldots, v'_d$ and a perfect $\epsilon$-symmetric pairing $\varphi_{d,d}$ by extending
\[
 \varphi_{d,d}(v_i,v_j)= \varphi_{d,d}(v'_i,v'_j)=0 \quad \text{and} \quad 
 \varphi_{d,d}(v_i,v_j') = \begin{cases}
    (-1)^i, & i+j = d+1 \\
    0, & \text{otherwise}
    \end{cases}
\]
Define a nilpotent $N_{d,d} \in \End(M(d,d))$ by $N_{d,d} v_i = v_{i-1}$ and $N_{d,d} v'_i = v'_{i-1}$ for $1 < i \leq d$, and $N_{d,d} v_1 = N_{d,d} v'_1 = 0$.  
\end{defn}

Note that the pairing $\varphi_{d,d}$ can be symmetric or alternating depending on the parity of $d$.
It is straightforward to verify the pairings are perfect and that $N_d$ (respectively $N_{d,d}$) is skew with respect to $\varphi_{d}$ (respectively $\varphi{d,d}$) in the sense that for $v,w \in M(d)$ we have
\[
 \varphi_d( N _d v, w) = -\varphi_d(v,N_d w).
\]

Given an admissible partition $\sigma : m = m_1 + m_2 + \ldots +m_r$, we will construct a free $\O$-module of rank $m$ with an $\epsilon$-symmetric perfect pairing and a nilpotent endomorphism that is skew with respect to the pairing such that the Jordan block structure of nilpotent endomorphism in geometric fibers is given by $\sigma$.  Let $n_i(\sigma) = \# \{ j : m_j = i \}$.  
\begin{itemize}
 \item If $G = \Orth_m$ then $n_i(\sigma)$ is even for even $i$, so we can define \[
   M_\sigma = \bigoplus_{i \, \text{odd}} M(i)^{\oplus n_i(\sigma)} \oplus \bigoplus_{i \, \text{even}} M(i,i) ^{\oplus n_i(\sigma)/2}.                             
                                \]
\item  If $G = \Sp_m$ then $n_i(\sigma)$ is even for odd $i$, so we can define
\[
  M_\sigma = \bigoplus_{i \, \text{odd}} M(i,i) ^{\oplus n_i(\sigma)/2} \oplus \bigoplus_{i \, \text{even}} M(i)^{\oplus n_i(\sigma)} .    
\]
\end{itemize}
Let $\varphi_\sigma$ and $N_\sigma$ denote the pairing and nilpotent endomorphism defined by the pairing and nilpotent endomorphism on each piece using Definition~\ref{defn:nilpotentpieces}.  In all cases, $M_\sigma$ is a free $\O$-module of rank $m$.  For each $\sigma,$ let $G_\sigma$ be the automorphism scheme $\underline{\Aut}(M_\sigma,\varphi_\sigma)$, so for an algebraically closed field $F$ over $\O$ we have an isomorphism $(G_\sigma)_F \simeq G_F$ well-defined up to $G(F)$-conjugation by using $F$-linear isomorphisms $(M_\sigma,\varphi_\sigma)_F \simeq (F^m,\varphi_{\on{std}})$.  

\begin{lem} \label{lem:geometricpts}
For all admissible partitions of $m$, the specializations of the $N_\sigma$ at geometric points $\xi$ of $\Spec \O$ constitute a set of representatives for the nilpotent orbits of $G_\xi$, and the specializations lie in the orbit corresponding to $\sigma$.
\end{lem}

\begin{proof}
The set of admissible partitions of $m$ is in bijection with the set of nilpotent orbits over any algebraically closed field.  The $N_\sigma$ we constructed are integral versions of the representatives constructed in \cite[\S1.7]{jantzen04}.
\end{proof}

Let $e_1,e_2, \ldots e_{m}$ be the standard basis for $\O^m$.  The elements $e_i$ and $e_{m+1-i}$ pair non-trivially under the standard pairing.  When $m = 2n+1$, $e_{n+1}$ pairs non-trivially with itself under the standard pairing.  We now relate the standard pairings to the pairings $\varphi_{\sigma}$.

\begin{prop} \label{prop:symorthreps}
Suppose that $\sqrt{-1} , \sqrt{2} \in \O^\times$.  Then $\varphi_{\sigma}$ is equivalent to the standard pairing over $\O$.  There exists an $\O$-basis $\{v_i\}$ of $\O^m$ with respect to which the pairing is given by $\varphi_\sigma$ and $N_\sigma$ satisfies the condition in \eqref{eq:ncondition} for $G = \Sp_m$ or $G = \Orth_m$.
\end{prop}

\begin{proof}
The standard pairings are very similar to $\varphi_\sigma$.  In the case of $\Sp_{m}$, each basis vector pairs trivially against all but one other basis vector, with which it pairs as $\pm 1$.  So after reordering the basis, $\varphi_\sigma$ is the standard pairing.  The case of $\Orth_m$ is slightly more complicated.  Let $\sigma : m = m_1 + m_2 + \ldots + m_r$ be an admissible partition.  The construction of $M_\sigma$ and $\varphi_\sigma$ gives a basis $\{v_{i,j}\}$ where $1 \leq i \leq r$ and $1 \leq j \leq m_i$.  From the construction of $\varphi_\sigma$, we see that $v_{i,j}$ pairs trivially against all basis vectors except for $v_{i,m_i+1-j}$.  So as long as $2j \neq m_i+1 $, we obtain a pair of basis vectors which are orthogonal to all others and which pair to $\pm 1$.  For each odd $m_i$, the vector $v_{i,(m_i+1)/2}$ pairs non-trivially with itself.  The standard pairing with respect to the basis $e_i$ has such a vector only when $m = 2n+1$ and then only for one $e_i$.  

We must change the basis over $\O$ so that $\varphi_\sigma$ becomes the standard symmetric pairing.  Let $v = v_{i,(m_i+1)/2}$ and $v' = v_{j,(m_j+1)/2}$ be two distinct vectors which pair non-trivially with themselves.  In particular, $\varphi_\sigma(v,v) = (-1)^{(m_i+1)/2} := \eta$ and $\varphi_\sigma(v',v') = (-1)^{(m_j+1)/2} := \eta'$.  Define
\[
 w = \frac{\sqrt{ \eta} v - \sqrt{- \eta'} v'}{\sqrt{2}} \quad \text{and} \quad w' = \frac{\sqrt{\eta} v + \sqrt{- \eta'} v'}{\sqrt{2}}.
\]
Then we see that $\varphi_\sigma(w,w)=0 = \varphi_\sigma(w',w')$ and $\varphi_\sigma(w,w') = 1$.  Making this change of variable over $\O$ (which requires $\sqrt{-1}, \sqrt{2} \in \O^\times$), we have reduced the number of basis vectors which pair non-trivially with themselves by two, and produced a new pair of basis vectors orthogonal to the others and which pair to $1$.  By induction, we may therefore pick a basis $v'_1 ,\ldots , v'_m$ for which at most one basis vector pairs non-trivially with itself under $\varphi_\sigma$.  After re-ordering, we may further assume that $\varphi_\sigma(v'_i, v'_j)=0$ unless $i +j = m+1$, in which case $\varphi_\sigma(v'_i,v'_j) = \pm 1$.  Suppose $j = m+1-i$.  If $i \neq j$, by scaling $v'_i$ we may assume that $\varphi_\sigma(v'_i,v'_j) = 1$.  If $i=j$, we already know that $\varphi_\sigma(v'_i,v'_j) =1$.  With respect to this basis, $\varphi_\sigma$ is the standard pairing.

The last statement immediately follows from Lemma~\ref{lem:geometricpts}.
\end{proof}

\subsection{Pure Nilpotent Lifts} \label{sec:purenilplift}
For a nilpotent element $\Nbar \in \g_k$ of type $\sigma \in \C$, we will define the notion of a \emph{pure nilpotent} lift of $\Nbar$ in $\g_k$ and study the space of such lifts, assuming 
there exists $N_\sigma \in \g$ lifting $\Nbar$ such that $(N_\sigma)_{\overline{K}} \in O_{\overline{K},\sigma}$ and such that $Z_G(N_\sigma)$ is smooth over $\O$.  

\begin{remark} \label{rmk:43}
 \S\ref{sec:integralrep} shows that for any nilpotent $\Nbar \in \g_k$, there exists $N'_\sigma \in \g$ such that $(N'_\sigma)_{\overline{k}} \in O_{\kbar,\sigma}$ and such that $(N'_\sigma)_k$ and $\Nbar$ are $G(\kbar)$-conjugate.  We will address the $\O$-smoothness of $Z_G(N_\sigma)$ in \S\ref{sec:centralizers}, especially Proposition~\ref{prop:similitudecases}.  Then $(N'_\sigma)_k$ and $\Nbar$ are conjugate by $\overline{g} \in G(k')$ for some finite extension $k'/k$.  Lift $\overline{g}$ to an element $g \in G(\O')$ for a Henselian discrete valuation ring local over $\O$ and having residue field $k'$.  The element $N_\sigma := g N'_ \sigma g^{-1} \in \g_{\O'}$ reduces to $\Nbar_{k'}$ and has the required properties.  So the above hypothesis is satisfied after a finite flat local extension of $\O$.
\end{remark}

\begin{defn} \label{defn:purenilp}
Fix an $N_\sigma \in \g$ lifting $\Nbar$ such that $(N_\sigma)_{\overline{K}} \in O_{\overline{K},\sigma}$ and such that $Z_G(N_\sigma)$ is smooth over $\O$.
Define the functor $\Nil_{\Nbar} : \C_\O \to \Sets$ by
\[
 \Nil_{\Nbar}(R) = \{ N \in \g_R | \Ad_G(g) (N_\sigma) = N \, \text{for some} \, g \in \Ghat(R) \}.
\]
Call these $N \in \Nil_{\Nbar}(R)$ the \emph{pure nilpotents} lifting $\Nbar$.  
\end{defn}

This is obviously a subfunctor of the formal neighborhood of $\Nbar$ in the affine space $\underline{\g}$ over $\O$ attached to $\g$.  The key to analyzing $\Nil_{\Nbar}$ is that $Z_{G_R}(N)$ is smooth over $R$ since $Z_G(N_\sigma)$ is $\O$-smooth and $N$ is in the $G$-orbit of $(N_\sigma)_R$.  To ease notation below, we shall write $gN g^{-1}$ rather than $\Ad_G(g)(N)$ for $g \in \Ghat(R)$.

\begin{lem}
Assuming $Z_G(N_\sigma)$ is $\O$-smooth, the functor $\Nil_{\Nbar}$ is pro-representable.
\end{lem}

\begin{proof}
We will use Schlessinger's criterion to check pro-representability.  As $\Nil_{\Nbar}$ is a subfunctor of the formal neighborhood of the scheme $\underline{\g}$ at $\Nbar$, the only condition to check is the analogue of Definition~\ref{defn:deformationcondition}\eqref{defcon2}:  given a Cartesian diagram in $\C_\O$
 \[
  \xymatrix{
  R_1 \times_{R_0} R_2 \ar[r]^{\pi_2} \ar[d]^{\pi_1} & R_2\ar[d]\\
  R_1 \ar[r] & R_0
  }
 \]
and $N_i \in \Nil_{\Nbar}(R_i)$ such that $N_1$ and $N_2$ reduce to $N_0$, we want to check that $N_1 \times N_2 \in \Nil_{\Nbar}(R_1 \times_{R_0} R_2)$.  By definition, there exists $g_1 \in \Ghat( R_1)$ and $g_2 \in \Ghat( R_2)$ such that $N_1 = g_1 N_\sigma g_1^{-1}$ and $N_2 = g_2 N_\sigma g_2^{-1}$.  Consider the element $g_1 g_2^{-1} \in \Ghat(R_0)$.  Observe that
\[
 g_1 g_2^{-1} N_\sigma g_2 g_1^{-1} = g_1 N_0 g_1^{-1} = N_\sigma \in \g_{R_0}.
\]
In particular, $g_1 g_2^{-1} \in Z_G(N_\sigma)(R_0)$.  The extension $R_2 \to R_0$ has nilpotent kernel, so as $Z_G(N_\sigma)$ is smooth over $\O$ there exists $h \in Z_G(N_\sigma) (R_2)$ lifting $g_1 g_2^{-1}$.  The element
\[
 (g_1, h g_2) \in  R_1 \times_{R_0}  R_2
\]
conjugates $N_1 \times N_2$ to $N_\sigma$.  Hence $N_1 \times N_2 \in \Nil_{\Nbar}(R_1 \times_{R_0} R_2)$.
\end{proof}

\begin{lem} \label{lem:nilpliftable}
The functor $\Nil_{\Nbar}$ is formally smooth, in the sense that for a small surjection $R_2 \to R_1$ of coefficient $\O$-algebras the map
\[
 \Nil_{\Nbar}(R_2) \to \Nil_{\Nbar}(R_1)
\]
is surjective.  Moreover, when $Z_G(N_\sigma)$ is $\O$-smooth and $\Nil_{\Nbar}$ is representable, it has relative dimension $\dim G_k - \dim Z_{G_k}(N_k)$ over $\O$.
\end{lem}

\begin{proof}
 Given $N \in \Nil_{\Nbar}(R_1)$, there exists $g \in \Ghat( R_1)$ such that $g N g^{-1} = N_\sigma$.  As $G$ is smooth over $\O$, we may find $g' \in \Ghat(R_2)$ lifting $g$.  Then $(g')^{-1} N_\sigma g'$ is a lift of $N$ to $R_2$.  From its definition, the tangent space to $\Nil_{\Nbar}$ is $\g_k / \zfrak_\g(N_k)$, so the formally smooth $\Nil_{\Nbar}$ has relative dimension $\dim G_k - \dim Z_{G_k}(N_k)$ since $Z_G(N)$ is $\O$-smooth.
\end{proof}

\begin{lem} \label{lem:inverseniln}
Suppose that $A$ is a complete local Noetherian $\O$-algebra with residue field $k$.  Under the assumption that $Z_G(N_\sigma)$ is $\O$-smooth, the inverse limit $\inverselimit \Nil_{\Nbar} (A/\m_A^n)$ equals $\{ N \in \g_A : N = g N_\sigma g^{-1}  \text{ for some } g \in G(A) \}$.
\end{lem}

\begin{proof}
This is immediate since $A$ is complete and $\Nil_{\Nbar}$ is representable by an affine scheme.
%
%
\end{proof}

\begin{remark}
If we had defined $\Nil_{\Nbar}$ on the larger category $\Chat_\O$ in the obvious way, Lemma~\ref{lem:inverseniln} would say that $\Nil_{\Nbar}$ is continuous.
\end{remark}

\begin{remark}
 There is no problem generalizing Definition~\ref{defn:purenilp} to any almost-reductive group, using a construction of $N_\sigma$ using root data as discussed in Remark~\ref{rmk:rootreps}.  However, we only establish the smoothness of $Z_G(N_\sigma)$ for classical groups.
\end{remark}

\begin{remark}
One can define a scheme-theoretic ``nilpotent cone'' over $\O$ as the vanishing locus of the ideal of non-constant homogeneous $G$-invariant polynomials on $\g$.  The arguments in this section could be rephrased as constructing a formal scheme of pure nilpotents inside the formal neighborhood of $\Nbar$ in $\underline{\g}$.  A natural question is whether there is a broader notion of pure nilpotents that gives a locally closed subscheme of the scheme-theoretic nilpotent cone.  For instance, for $N, N' \in \g$, if their images in $\g_K$ and $\g_k$ are nilpotent in orbits with the same combinatorial parameters, are $N$ and $N'$ conjugate under $G$ over a discrete valuation ring local over $\O$?

When $G = \GL_n$, this has been explored by Taylor in the course of constructing local deformation conditions \cite[Lemma 2.5]{taylor08}.  The method uses the explicit description of the orbit closures given by specifying the Jordan canonical form to define an analogue of the orbit closures over $\O$.  It would be interesting to find a way to do so more generally.
\end{remark}

\section{Smoothness of Centralizers of Pure Nilpotents} \label{sec:centralizers}

In order for the functor $\Nil_{\Nbar}$ to be representable, we need that $Z_G(N_\sigma)$ is $\O$-smooth.  
Recall that for $N \in \g$, the scheme-theoretic centralizer $Z_G(N)$ represents the functor
\[
 R \mapsto \{g \in G(R) : \Ad_G(g) N_R = N_R \}
\]
for $\O$-algebras $R$.  We will study the centralizer $Z_G(N_\sigma)$ in more detail where $N_\sigma \in \g$ is an element satisfying \eqref{eq:ncondition}.  In particular, this centralizer will be shown to be smooth when $G$ is symplectic or orthogonal.  We first review the known theory over fields, and then develop and apply a technique to deduce smoothness over $\O$ (i.e. $\O$-flatness) from the known smoothness in the field case.

\subsection{Centralizers over Fields} \label{sec:centralizerssmooth}
 In this section, let $k$ be an algebraically closed field, $G$ be a connected reductive group over $k$, and $N$ a nilpotent element of $\g = \Lie G$.  
As the formation of the scheme-theoretic centralizer commutes with base change, smoothness results for $Z_G(N)$ over $k$ will imply such results over general fields (not necessarily algebraically closed).  

The group scheme $Z_G(N)$ is the fiber over $0 \in \g$ of the composition
\[
 G \overset{\Ad_G} \longrightarrow \GL(\g) \overset{T \mapsto T N - N} \longrightarrow \g.
 \]
Hence $\Lie Z_G(N)$ is the kernel of
\[
 \g \overset{\ad_\g} \longrightarrow \End(\g) \overset{T \mapsto T N }\longrightarrow \g
\]
which is the Lie algebra centralizer $\zfrak_\g(N)$.  

\begin{remark}
In references using the language of varieties rather than schemes (such as \cite{jantzen04}), $Z_G(N)$ is usually \emph{defined} via its geometric points and hence is reduced and smooth, so the condition that the scheme $Z_G(N)$ is smooth becomes the condition that the variety $Z_G(N)$ has Lie algebra $\mathfrak{z}_\g(N)$.
\end{remark}

In a wide range of situations, all nilpotent centralizers are smooth.  A direct calculation shows that this holds for $G = \GL_n$ (see \cite[\S2.3]{jantzen04}), and a criterion of Richardson \cite[Theorem 2.5]{jantzen04} can be applied to show:

\begin{fact} \label{fact:gspcentralizers}
 If $G$ is an orthogonal or symplectic (similitude) group, any nilpotent centralizer is smooth over $k$.  
\end{fact}

\begin{remark}
Suppose $Z_G(N)$ is smooth over $k$ and $p$ is good for $G$.  The classification of nilpotent orbits is independent of $p$, as are their dimensions, so the dimension of $Z_G(N)$ is independent of $p$ as well.
\end{remark}

\subsection{Checking Flatness over a Dedekind Base} \label{sec:technical}

We want to analyze smoothness of centralizers in the relative setting (especially over $\Spec \O$).  If $Z_G(N_\sigma) \to \Spec \O$ is flat and the special and generic fibers are smooth then $Z_G(N)$ is smooth over $\O$.  The following lemma gives a way to check that a morphism to a Dedekind scheme is flat.

\begin{lem} \label{lem:flatnesscriterion}
Let $f : X \to S$ be finite type for a connected Dedekind scheme $S$.  Then $f$ is flat provided the following all hold:
\begin{enumerate}
 \item for each $s \in S$, $X_s$ is reduced and non-empty;
 \item for each $s \in S$, $X_s$ is equidimensional with dimension independent of $s$;
 \item there are sections $\{\sigma_i \in X(S)\}$ to $f$ such that for every irreducible component of a fiber above a closed point, there is a section $\sigma_i$ which meets the fiber only in that component.
\end{enumerate}
\end{lem}

\begin{remark}
This lemma is a modification of \cite[Proposition 6.1]{gy03} to allow multiple irreducible components in the fibers.
\end{remark}

\begin{proof}
 It suffices to prove the result when $S = \Spec (A)$ for $A$ a discrete valuation ring with uniformizer $\pi$.  Let $X_\eta$ be the generic fiber and $X_s$ the special fiber.  Consider the schematic closure $\imath : X' \hookrightarrow X$ of the generic fiber.  The scheme $X'$ is flat over $\Spec(A)$ since flatness is equivalent to being torsion-free over a discrete valuation ring, and there is an exact sequence
 \begin{equation} \label{eq:idealsheaf}
  0 \to J \to \O_X \to \imath_* \O_{X'} \to 0
 \end{equation}
 where $J$ is a coherent sheaf killed by a power of $\pi$.  
We will show that $\imath$ is an isomorphism by analyzing the special fiber.

First, we claim that the dimension of each irreducible component on the special fiber of $X'$ is the same as the dimension of the equidimensional $X_\eta$.  We will get this from flatness of $X'$.  The generic fiber of $X'$ is $X_\eta$, which is equidimensional and non-empty by hypothesis.  Furthermore, $X'$ is the union of the closures $Z_i$ of the reduced irreducible components $X_{\eta,i}$ of $X_\eta$, and each $Z_i$ is $A$-flat with integral $\eta$-fiber, hence integral.  We just need to analyze the dimension of irreducible components of $(Z_i)_s$ when $(Z_i)_s \neq \emptyset$.  Since $Z_i$ is integral, we can apply \cite[Theorem 15.1, 15.5]{crt} to such $Z_i$ to conclude that the dimension of each irreducible component of the special fiber of $X'$ is the same as the dimension of the generic fiber.

Observe that the sections $\sigma_i$ factor through the closed subscheme $X' \subset X$, as we can check this on the generic fiber since $X'$ is $A$-flat.  Thus $X'$ meets every irreducible component of $X_s$ away from the other irreducible components of $X_s$.  We would have that $|X_s'| = |X_s|$ if $X'_s$ is equidimensional of the same dimension as the equidimensional $X_s$.  We have shown the dimension of any irreducible component in $X'_s$ is the same dimension as the common dimension of irreducible components of the generic fiber $X_\eta$ of $X'$.  By hypothesis, the dimension of any irreducible component of the generic fiber of $X$ is the same as the dimension of any irreducible component of the special fiber of $X$.  Thus the dimension of any irreducible component of $X'_s$ is the same as the dimension of each irreducible component of $X_s$, giving that $|X_s'| = |X_s|$.  As $X_s$ is reduced, this forces
 $\imath_s : X'_s \into X_s$ to be an isomorphism.

Now tensoring $\eqref{eq:idealsheaf}$ with the residue field of $A$ gives an exact sequence
\[
 0 \to J/\pi J   \to \O_{X,s} \to \imath_* \O_{X',s} \to 0
\]
because $\O_{X'}$ is $A$-flat.  But $J/\pi J = 0$ as $\imath_s$ is an isomorphism.  Hence $J = \pi J = \pi^2 J = \ldots = \pi^n J = 0$ for $n$ large, so $X = X'$ is flat over $A$.
\end{proof}

\begin{cor} \label{cor:smoothnesscriterion}
In the situation of the lemma, if the fibers are also smooth then $X$ is smooth.
\end{cor}

\begin{proof}
For a flat morphism of finite type between Noetherian schemes, smoothness of all fibers is equivalent to smoothness of the morphism.
\end{proof}

\subsection{Centralizers for Orthogonal and Symplectic Groups} \label{sec:centclassical}

To apply Corollary~\ref{cor:smoothnesscriterion}, we need information about the component group of centralizers of nilpotents.  For $\GL_n$ over a field, all such centralizers are connected.  For symplectic and orthogonal groups, there is an explicit description of $Z(N_\sigma)$ where $N_\sigma$ is the nilpotent constructed in \S\ref{sec:nilprep}.  We continue the notation of that section: $G$ is $\Sp_m$ or $\Orth_m$ (with $m \geq 4$) over a discrete valuation ring $\O$ with a residue field $k$ of good characteristic $p\neq 2$.  

Let $\sigma : m_1 + \ldots + m_r$ be an admissible partition of $m$.  We assume that $\O$ is large enough so that Proposition~\ref{prop:symorthreps} holds, and take $N := N_\sigma$.  Then there exists elements  $v_1 ,\ldots ,v_r \in M:= \O^m$ such that
\[
 v_1, N v_1, \ldots , N^{m_1-1} v_1, v_2 , N v_2 , \ldots , N^{m_r-1} v_r
\]
is a basis for $M$.  Furthermore, $N^{m_i} v_i = 0$ for $i =1 , \ldots r$, and the pairing between basis elements is given by $\varphi:= \varphi_\sigma$.  In particular, each $v_i$ pairs non-trivially with only one other basis element $X^{d_i-1} v_{i^*}$, for some $i^* \in \{1,\ldots ,r\}$. 

To understand the $G$-centralizer of $N$, we construct an associated grading of $M$ as in \cite[\S3.3,3.4]{jantzen04}.  This is motivated by the Jacobson-Morosov theory of $\mathfrak{sl}_2$-triples over a field of sufficiently large characteristic, but for symplectic and orthogonal groups it is constructed by hand in characteristic $p \neq 2$ below.  

\begin{remark}
Let $M_k = M \tensor{\O} k = k^m$.
 Every nilpotent $X \in \End(M_k)$ gives a filtration (and grading) of $M_k$ defined by $\Fil^i = \ker (X^i)$.  For $\GL_n$, this is a nice filtration and is used in \cite{cht08} to define the minimally ramified deformation condition for $\GL_n$.  However, this filtration need not be isotropic with respect to the pairing, so we will construct a nicer grading associated to $X$.
\end{remark}


\begin{defn}
Let $M(s)$ be the span of $N^j v_i$ for all $i$ and $j$ such that $s = 2j+1-m_i$.  We set $M^{(s)} = \bigoplus_{t \geq s} M(t)$, and also define $L(s)$ to be the span of $\{v_i : v_i \in M(s)\}$.
\end{defn}


We now record some elementary properties of the preceding construction; all are routine to check, and the analogous proofs over a field may be found in \cite[\S3.4]{jantzen04}.  We have that $M = \bigoplus_s M(s)$, and 
$$v_i \in M(-(d_i-1)), N v_i \in M(-(d_i-1) +2), \ldots , N^{d_i-1} v_r \in M(d_r-1).$$
Furthermore, we know $N \cdot M(s) \subset M(s+2)$ and $M(s) = N \cdot M(s-2) \oplus L(s)$ for $s \leq 0$.

The dimension of $M(s)$ is $m_s(\sigma) := \# \{ j : d_j -1 \geq |s|\}$.  The dimension of $L(s)$ equals $l_{s}(\sigma) := m_{s+1}(\sigma) - m_s(\sigma) $.  Furthermore, the pairing $\varphi$ interacts well with the grading: a computation with basis elements gives that 
$\varphi( M(s), M(t)) \neq 0$ implies $s+t=0$.

The above grading on $M$ corresponds to the one-parameter subgroup $\lambda : \Gm \to G$ for which the action of $t \in \Gm$ on $M(s)$ is given by scaling by $t^s$.  
The dynamic method (see \cite[\S2.1]{pred15}) associates to $\lambda$ a parabolic subgroup $P_G(\lambda)$ with Levi $Z_G(\lambda)$.  Define $C_N$ and $U_N$ to be the scheme-theoretic intersections
\begin{align*}
 C_N &= Z_G(N) \cap Z_G(\lambda) = \{g \in Z_G(N) : g M(i) = M(i) \text{ for all } i \} \\
 U_N &= Z_G(N) \cap U_G(\lambda) = \{ g \in Z_G(N) : (g-1) M^{(i)} \subset M^{(i+1)} \text{ for all } i\}.
\end{align*}

\begin{fact} \label{fact:centralizerstructure}
 The group-scheme $Z_G(N)_k$ is a semi-direct product of $(C_N)_k$ and the smooth connected unipotent subgroup $(U_N)_k$.  In particular, the connected components of $Z_G(N)_k$ are the same as the connected components of $(C_N)_k$.
\end{fact}

\begin{remark}
This is \cite[Proposition 3.12]{jantzen04}.  The existence of $\lambda$ and this decomposition is not specific to symplectic and orthogonal groups \cite[Proposition 5.10]{jantzen04}.
\end{remark}

We finally give a concrete description of $C_N$.  We first define a pairing on $L(s)$.  Recall that the space $L(s)$ of ``lowest weight vectors'' in $M(s)$ has basis $\{ v_i : 1-d_i = s\}$.  We define a pairing on $L(s)$ by
\[
 \psi_s(v,w) = \varphi(v,N^{-s} w).
\]
A direct calculation shows that $\psi_s$ is non-degenerate and that $\psi_s$ is symmetric if $(-1)^{s} = \epsilon$ and is alternating if $(-1)^{s} = -\epsilon$ \cite[\S3.7]{jantzen04}.  

A point of $C_N$ preserves the grading on $M$, and since it commutes with the ``raising operator'' $N$ its action on $M$ is determined by its action on the space $L(s)$ of  ``lowest weight vectors'' in $M(s)$.  So the following fact is no surprise.

\begin{prop}
There is an isomorphism of algebraic groups
\[
 C_N \simeq \prod_{s \leq 0} \underline{\Aut}(L(s),\psi_s)
\]
\end{prop}

The corresponding statement over a field is \cite[\S3.8 Proposition 2, 3]{jantzen04}: the proof is the same.

\begin{example} \label{ex:spcentralizer}
 Let $G = \Sp_{m}$.  Unraveling when $\psi_s$ is symmetric or alternating, we see that
 \[
  C_N \simeq \prod_{s \leq 0; s \, \text{odd}} \Orth(L(s),\psi_s) \times \prod_{s \leq 0; s \, \text{even}} \Sp(L(s),\psi_s).
 \]
 The special fibers of the symplectic factors are connected, while the orthogonal factors have two connected components in the special fiber.
 Thus there are $2^t$ connected components, where $t$ is the number of odd $s$ for which $L(s) \neq 0$.  For each component, there is a section $g \in C_N(\O)$ meeting that component corresponding to a choice of $\pm \on{Id} \in \Orth(L(s),\psi_s)$ for each odd $s$ with $L(s) \neq 0$.   The connected components of $Z_G(N)$ are the same as those for $C_N$ by Fact~\ref{fact:centralizerstructure}.
\end{example}

\begin{example} \label{ex:socentralizer}
 Let $G = \Orth_m$.  We likewise see that
  \[
  C_N \simeq \prod_{s \leq 0; s \, \text{even}} \Orth(L(s),\psi_s) \times \prod_{s \leq 0; s \, \text{odd}} \Sp(L(s),\psi_s) .
 \]
Again there are $2^t$ connected components of $Z_G(N)$, where $t$ is the number of even $s$ for which $L(s) \neq 0$. 
 
 Now suppose that $G = \SO_m$.  The elements $N$ we considered in this section are representatives for some of the nilpotent orbits of $\SO_m$.  The group $C_N$ has the same structure as for $G = \Orth_m$, except we require that the overall determinant be $1$; this has $2^{t-1}$ connected components.  Though $\SO_m$ has more nilpotent orbits than $\Orth_m$, according to Remark~\ref{rmk:somorbits} their representatives are conjugate by an element $\Orth_m(k)$ with determinant $-1$ to the representatives constructed in Proposition~\ref{prop:symorthreps}.  This shows that there are sections $g \in C_N(\O)$ meeting each component.
\end{example}

\begin{remark} \label{rmk:phi}
Suppose $q$ is a square in $\O^\times$.  For use in the proof of Proposition~\ref{prop:smoothnessmrtame}, we need the existence of an element $\Phi \in G(\O)$ such that $\ad_G(\Phi) N_\sigma = q N_\sigma$.  If $\alpha^2 = q$, taking $\Phi = \lambda(\alpha)$ would work: $\Phi$ would scale $N_\sigma^j v_i \in  M(s)$ by $\alpha^s$, and $N_\sigma$ increases the degree by $2$.

This $\Phi$ is a version for symplectic and orthogonal groups of the diagonal matrix denoted $\Phi(\sigma,a,q)$ whose diagonal entries are increasing powers of $q$ used in \cite[\S2.3]{taylor08}.  There it is checked that $\ad_G(\Phi(\sigma,a,q)) N_\sigma = q N_\sigma$ where $N_\sigma$ is the nilpotent representative in Jordan canonical form considered in Example~\ref{ex:glnrep} for the partition $\sigma$ of $m$.
\end{remark}

\subsection{Smoothness of Centralizers} 
We now return the case when $G$ is an almost-reductive group over a discrete valuation ring $\O$ with residue field $k$ of \emph{very good} characteristic $p>0$.  Suppose we are given an integral representative $N=N_\sigma \in \g := \Lie G$ for the nilpotent orbit on geometric fibers corresponding to $\sigma \in \C$ as in \eqref{eq:ncondition}: that is, an element such that
\begin{equation*}
N_k \in O_{\kbar,\sigma} \, \text{ and } \, N_K \in O_{\overline{K},\sigma}.  
\end{equation*}
Proposition~\ref{prop:symorthreps} provides such $N$ in symplectic and orthogonal cases when $\sqrt{-1}, \sqrt{2} \in \O^\times$.  
We wish to check that the $Z_G(N)$ is smooth over $\O$.  This $N$ satisfies
\begin{equation} \label{eq:smoothfibers}
Z_{G_K}(N_K)  \text{ and }  Z_{G_k}(N_k) \text{ are smooth of the same dimension.}
\end{equation}

\begin{remark}
Some assumption on $N$ is essential.  Otherwise $N_{\overline{K}}$ and $N_{\overline{k}}$ can lie in different nilpotent orbits (in terms of the combinatorial characteristic-free classification of geometric orbits), and so $Z_{G_K}(N_K)$ and $Z_{G_k}(N_k)$ could have different dimensions, in which case $Z_G(N)$ cannot be $\O$-flat.  An example of this is the element $N_2$ in Example~\ref{ex:changeorbit}.
\end{remark}

Now we wish to check the conditions necessary to apply Corollary~\ref{cor:smoothnesscriterion}.  
We define
\[
  A(N) = Z_{G_k}(N_k)(k) / Z_{G_k}(N_k)^\circ(k),
\]
and study when the following holds:
\begin{equation} \label{dagger}
 \text{each element of } A(N)  \text{ arises from some }  s \in Z_G(N)(\O). 
\end{equation}
Note that this checks the criterion in Corollary~\ref{cor:smoothnesscriterion} as $$Z_{G_k}(N_k)(k) / Z_{G_k}(N_k)^\circ(k) =  \left( Z_{G_{k}}(N_k)/Z_{G_k}(N_k)^\circ \right) (k)$$ by Lang's theorem, and since the irreducible components of $Z_{G_k}(N_K)$ are the same as connected components by smoothness.

We are free to make a local flat extension of $\O$, as it suffices to check flatness after such an extension.  In particular, it suffices to check \eqref{dagger} when $\O$ is Henselian and $k$ is algebraically closed, and $\sqrt{-1}, \sqrt{2} \in \O$.  Examples \ref{ex:spcentralizer} and \ref{ex:socentralizer} give such sections when $G = \Sp_m$ or $G=\SO_m$.  We will use these cases to get a result for similitude groups. 

Let $\pi : \widetilde{G'} \to G'$ be the simply connected central cover of the derived group $G'$ over $\O$.  
As $p$ is very good, $\widetilde{G'}$ and $G'$ have isomorphic Lie algebras via $\pi$ and $\Lie G'$ is a direct factor of $\Lie G$ with complement $\Lie(Z_G)$, so we may abuse notation and view $N$ as an element of all of these Lie algebras over $\O$.

Let $S$ be a (split) maximal central torus in $G$.  Consider the isogeny $S \times \widetilde{G'} \to G$.  As $S$ acts trivially on $N$, we see that $S \times Z_{\widetilde{G'}}(N)$ is the preimage of $Z_{G'}(N)$ under this isogeny.  As $p$ is very good for $G$, we obtain finite \'{e}tale surjections
\[
  Z_{\widetilde{G'}}(N) \to Z_{G'}(N) \quad \text{and} \quad S \times Z_{\widetilde{G'}}(N) \to Z_G(N)
 \]
over $\O$.

\begin{lem} \label{lem:reduction1}
The condition \eqref{dagger} holds for $\widetilde{G'}$ if and only if \eqref{dagger} holds for $G$.
\end{lem}

\begin{proof}
Assume $\widetilde{G'}$ satisfies \eqref{dagger}.  Pick a connected component $C$ of $Z_{G_k}(N_k)$.  The preimage of $C$ under $S \times Z_{\widetilde{G'}}(N) \to Z_G(N)$ is a union of $k$-fiber components of the form $S_k \times C'$ where $C'$ is a connected component of $Z_{\widetilde{G'_k}}(N_k)$.  By assumption, there exists $s \in Z_{\widetilde{G'}}(N)(\O)$ meeting any such $C'$.  The image of $(1,s)$ is a point of $Z_G(N)(\O)$ meeting $C$.

Conversely, assume $G$ satisfies \eqref{dagger}.  Pick a connected component $C'$ of $Z_{\widetilde{G'_k}}(N_k)$.  Under $S \times Z_{\widetilde{G'}}(N) \to Z_G(N)$, $S_k \times C'$ maps onto a connected component $C$ of $Z_{G_k}(N_k)$.  By assumption, there exists $s \in Z_G(N)(\O)$ such that $s_k \in C$.  As $k$ is algebraically closed, there is $s'_k \in (S \times Z_{\widetilde{G'}}(N))(k)$ lifting $s_k$ and lying in $C'$.  As $S \times Z_{\widetilde{G'}}(N) \to Z_G(N)$ is a finite \'{e}tale cover and $\O$ is Henselian, there exists $s' \in (S \times Z_{\widetilde{G'}}(N))(\O)$ lifting $s$ and reducing to $s'_k$.
\end{proof}

For example, this lets us pass between $\Sp_{2n}$ and $\GSp_{2n}$ by way of the projective symplectic group.

\begin{prop} \label{prop:similitudecases}
For $G$ a symplectic or orthogonal similitude group, and $N = N_\sigma \in \g$ the element satisfying \eqref{eq:ncondition} given by Proposition~\ref{prop:symorthreps} for an admissible partition $\sigma$, the centralizer $Z_G(N)$ is smooth over $\O$.  
\end{prop}

\begin{proof}
By Fact~\ref{fact:gspcentralizers}, $Z_{G_k}(N_k)$ and $Z_{G_K}(N_K)$ are smooth.  By the classification of nilpotent orbits over algebraically closed fields, the dimension of the orbit only depends on the combinatorial classification for the orbit in very good characteristic and in characteristic $0$, so these fibers are equidimensional of the same dimension.  By Corollary~\ref{cor:smoothnesscriterion}, it suffices to find $s \in Z_G(N)(\O)$ meeting any connected component of $Z_{G_k}(N_k)$.  

Using Lemma~\ref{lem:reduction1}, we reduce checking \eqref{dagger} to the cases of $\Sp_m$ and $\SO_m$, covered by Examples \ref{ex:spcentralizer} and \ref{ex:socentralizer}.
\end{proof}

\begin{remark} \label{rmk:glncomp}
Consider a nilpotent orbit of $\GL_n$ with representative $N$ given in Example~\ref{ex:glnrep}.  As $Z_{G_k}(N_k)$ is connected \cite[Proposition 3.10]{jantzen04}, the identity section shows \eqref{dagger} holds.  This shows $Z_G(N)$ is smooth.
\end{remark}

\begin{remark}
It is not hard to extend the above argument to work for groups such that all the irreducible factors of the root system are of classical type.  For the exceptional groups, one could find $N$ as in Remark~\ref{rmk:rootreps} and attempt to check \eqref{dagger} holds by hand (there are finitely many cases).  A conceptual approach would be preferable.
\end{remark}

\begin{remark}
McNinch analyzes the centralizer of an ``equidimensional nilpotent'' in \cite{mcninch08}.  An \emph{equidimensional nilpotent} is an element $N \in \g$ such that $N_K$ is nilpotent and the dimension of the special and generic fibers of $Z_G(N)$ are the same.  \cite[\S 5.2]{mcninch08} claims that such $Z_G(N)$ are $\O$-smooth because the fibers are smooth of the same dimension.  This deduction is incorrect: it relies on \cite[2.3.2]{mcninch08} which uses the wrong definition of an equidimensional morphism and thereby incorrectly applies \cite[Exp. II, Prop 2.3]{sga1}.

According to \cite[Exp. II, Prop 2.3]{sga1} (or \cite[\S13.3, 14.4.6, 15.2.3]{egaiv3}), for a Noetherian scheme $Y$, a morphism $f : X \to Y$ locally of finite type, and points $x \in X$ and $y = f(x)$ with $\mathscr{O}_y$ normal, $f$ is smooth at $x$ if and only if $f$ is equidimensional at $x$ and $f^{-1}(y)$ is smooth over $k(y)$ at $x$.  But by definition in \cite[13.3.2]{egaiv3}, an \emph{equidimensional} morphism is more than just a morphism all of whose fibers are of the same dimension (the condition checked in \cite[2.3.2]{mcninch08}): a locally finite type morphism $f$ is called equidimensional of dimension $d$ at $x \in X$ when there exists an open neighborhood $U$ of $x$ such that for every irreducible component $Z$ of $U$ through $x$, $f(Z)$ is dense in some irreducible component of $Y$ containing $y$ and for all $x' \in U$ the fiber $f^{-1} ( f(x')) \cap U$ has all irreducible components of dimension $d$.

This is much stronger than the fibers simply being of the same dimension.  To see the importance of the extra conditions, consider a discrete valuation ring $\O$ with field of fractions $K$ and residue field $k$, and the morphism from $X$, the disjoint union of $\Spec K$ and $\Spec k$, to $Y=\Spec \O$.  The fibers are of the same dimension (zero) and smooth but the morphism is not flat.  This morphism is also not equidimensional at $\Spec k$: the only irreducible component of $X$ containing $\Spec k$ is the point itself, with image the closed point of $\Spec \O$.  This is not dense in $\Spec \O$, the only irreducible component of the only open set containing the closed point of $\Spec \O$.

The smoothness of centralizers of an equidimensional pure nilpotent is important to proving the main results of \cite{mcninch08}.  In particular, the results in \S6 and \S7 in \cite{mcninch08} crucially rely on the smoothness of the centralizers of such nilpotents, leaving a gap in the proof of Theorem B in \cite{mcninch08} concerning the component group of centralizers.  The method we have discussed here reverses this, understanding the geometric component group well enough to \emph{produce} sufficiently many $\O$-valued points in order to deduce smoothness of the centralizer in classical cases in very good characteristic via Lemma~\ref{lem:flatnesscriterion}.  
\end{remark}

\section{Minimally Ramified Deformations: Tame Case} \label{sec:mrtame}

In this section, we will generalize the tamely ramified case of the minimally ramified deformation condition introduced in \cite[\S 2.4.4]{cht08} for $\GL_n$ to symplectic and orthogonal groups.  We also explain why another more immediate notion based on parabolic subgroups, giving the same deformation condition for $\GL_n$, is \emph{not liftable} in general (even for $\GSp_4$).  

\subsection{Passing between Unipotents and Nilpotents} \label{sec:passing}
As before, $G$ is either $\GSp_m$ or $\GO_m$ (or $\GL_m$ to recover the results of \cite[\S2.4.4]{cht08}) over the ring of integers $\O$ in a $p$-adic field with residue field $k$ of characteristic $p>0$. 
As always, we assume that $p$ is very good for $G_k$ (i.e. $p\neq 2$).  Let $\g = \Lie(G)$.

As in \S\ref{sec:purenilplift}, we work with a pure nilpotent $N_\sigma \in \g$ for which $Z_G(N_\sigma)$ is $\O$-smooth, $(N_\sigma)_{\overline{K}} \in O_{\overline{K},\sigma}$, and $(N_\sigma)_{\overline{k}} \in O_{\overline{k},\sigma}$.  Define $\Nbar := (N_\sigma)_k$. 
We studied deformations of $\Nbar$ in \S\ref{sec:purenilplift}, but will ultimately want to analyze deformations of Galois representations which take on unipotent values at certain elements of a local Galois group.  Thus, we need a way to pass between unipotent and nilpotent elements.  For classical groups, we can use a truncated version of the exponential and logarithm maps:

\begin{fact} \label{fact:expmap}
Suppose that $p \geq m$ and that $R$ is an $\O$-algebra.  If $A \in \Mat_m(R)$ has characteristic polynomial $x^m$ then
\[
 \exp(A) := 1 + A +  A^2/2 + \ldots + A^{m-1}/ (m-1)!
\]
has characteristic polynomial $(x-1)^m$.  If $B \in \Mat_m(R)$ has characteristic polynomial $(x-1)^m$ then
\[
 \log(B) := (B-1) - (B-1)^2/2 + \ldots + (-1)^m (B-1)^{m-1} / (m-1)
\]
has characteristic polynomial $x^m$.  Furthermore for $C \in \GL_m(R)$ and an integer $q$, we have 
\begin{multicols}{2}
\begin{itemize}
 \item $\exp(C A C^{-1}) = C \exp(A) C^{-1}$ \item $\log(C B C^{-1}) = C \log(B) C^{-1}$ \item $\log( \exp(A)) = A$ \item  $\exp(\log(B)) = B$ \item $\exp(q A) = \exp(A)^q$ \item $\log(B^q) = q \log(B)$
 \end{itemize}
\end{multicols}
\end{fact}

This is \cite[Lemma 2.4]{taylor08}.  The key idea is that because all the higher powers of $A$ and $B-1$ vanish and all of the denominators appearing are invertible as $p \geq m$, we can deduce these facts from results about the exponential and logarithm in characteristic zero.  

Suppose $J$ is the matrix for a perfect symmetric or alternating pairing over $R$.

\begin{cor} \label{cor:exppairing}
For $A$ and $B$ as in Fact~\ref{fact:expmap} with $\exp(A) = B$, $A^T J + J A =0$  if and only if $B^T J B = J$.
\end{cor}

\begin{proof}
Directly from the definitions we see that $\exp(A^T) = \exp(A)^T$.  Observe that 
$$\exp( J A J^{-1}) = J B J^{-1} \quad \text{and} \quad \exp(-A^T) = (B^T )^{-1}.$$
Thus $JA J^{-1} = - A^T$ if and only if $(B^T)^{-1} = J B J^{-1}$.
\end{proof}

We shall use this exponential map to convert pure nilpotents into unipotent elements.  Let $R$ be a coefficient ring over $\O$.  By Definition~\ref{defn:purenilp}, any pure nilpotent $N \in \Nil_{\Nbar}(R)$ is $G(R)$-conjugate to $N_\sigma$, so it has characteristic polynomial $x^m$.  Denoting the derived group of $G$ by $G'$, any nilpotent element of $\g$ lies in $(\g') = (\Lie G')$, so $N J + J N = 0$ (and not just $NJ + J N = \lambda J$ for some $\lambda \in \O$).  Thus, Corollary~\ref{cor:exppairing} shows that $\exp(N) \in G(R)$.  This gives an exponential map
\begin{equation} \label{eq:exponentialmap}
 \exp : \Nil_{\Nbar} \to G
\end{equation}
such that for $g \in \Ghat(R)$, $N \in \Nil_{\Nbar}(R)$, and $q \in \ZZ$ we have $\exp(q N) = \exp(N)^q$ and $g \exp(N) g^{-1} = \exp( \Ad_G(g) N)$.

\begin{remark}
This is a realization over $\O$ of a special case of the Springer isomorphism identifying the nilpotent and unipotent varieties in very good characteristic.  For later purposes, we will need that the identification is compatible with the multiplication in the sense that $\exp(q A) = \exp(A)^q$.  In the case of $\GL_m$, a Springer isomorphism that works in any characteristic is given by $X \to 1 +X$ for nilpotent $X$, but this is not compatible with multiplication.
\end{remark}

\subsection{Minimally Ramified Deformations} \label{sec:mrd}
As before, $G$ is $\GSp_m$, $\GO_m$ or $\GL_m$ over the ring of integers $\O$ of a $p$-adic field with residue field $k$ with $p > m$.  Let  $L$ be a finite extension of $\QQ_\ell$ (with $\ell \neq p$), and denote its absolute Galois group by $\Gamma_L$.  Consider a representation $\rhobar : \Gamma_L \to G(k)$.  We wish to define a (large) smooth deformation condition for $\rhobar$ generalizing the minimally ramified deformation condition for $\GL_n$ defined in \cite[\S2.4.4]{cht08}.  In this section we do so for a special class of tamely ramified representations.  This requires making an \'{e}tale local extension of $\O$, which will be harmless for our purposes.

Recall that $\Gamma_L^\tame$, the Galois group of the maximal tamely ramified extension of $L$, is isomorphic to the semi-direct product
\[
 \ZZhat \ltimes \prod_{p' \neq \ell} \ZZ_{p'}
\]
where $\ZZhat$ is generated by a Frobenius $\phi$ and the conjugation action by $\phi$ on the direct product is given by the cyclotomic character.  We consider representations of $\Gamma_L^\tame$ which factor through the quotient  $\ZZhat \ltimes \ZZ_p$.  
Picking a topological generator $\tau$ for $\ZZ_p$, the action is explicitly given by
\[
 \phi \tau \phi^{-1} = q \tau
\]
where $q$ is the size of the residue field of $L$.  Note $q$ is a power of $\ell$, so it is relatively prime to $p$.  This leads us to study representations of the group $T_q :=  \ZZhat \ltimes \ZZ_p$.

Let $\rhobar : T_q \to G(k)$ be such a representation.  We first claim that $\rhobar(\tau) \in G(k)$ is unipotent.  This element decomposes as a commuting product of semisimple and unipotent elements of $G(k)$.  The order of a semisimple element in $G(k)$ is prime to $p$, while by continuity there is an $r  \geq 0$ such that $\tau^{p^r} \in \ker(\rhobar)$.  Thus $\rhobar(\tau)$ is unipotent.

Informally, a deformation $\rho : T_q \to G(R)$ will be minimally ramified if $\rho(\tau)$ lies in the ``same'' unipotent orbit as $\rhobar(\tau)$.  To make this meaningful over an infinitesimal thickening of $k$, we shall use the notion of pure nilpotents as in Definition~\ref{defn:purenilp} since unipotence and unipotent orbits are not good notions when not over a field.
As $\Nbar := \log(\rhobar(\tau))$  is nilpotent, by Remark~\ref{rmk:43} after making an \'{e}tale local extension of $\O$ we may assume that there exists a pure nilpotent $N_\sigma \in \g$ lifting $\Nbar$ for which $Z_G(N_\sigma)$ is smooth.  Making a further extension if necessary, we may also assume that the unit $q \in \O^\times$ is a square.  We obtain an exponential map $\exp : \Nil_{\Nbar} \to G$ as in \eqref{eq:exponentialmap}.  

\begin{defn} \label{defn:minimallyramifiedtame}
Under our standing assumptions (collected as \ref{assumption1}-\ref{assumption4} in \S\ref{sec:notation}), for a coefficient ring $R$ over $\O$, a continuous lift $\rho : T_q \to G(R)$ of $\rhobar$ is \emph{minimally ramified} if $\rho(\tau) = \exp(N)$ for some $N \in \Nil_{\Nbar}(R)$.
\end{defn}

\begin{example} \label{ex:compatiblecht}
 Take $G = \GL_m$.  Then $X \mapsto 1_m +X$ gives an identification of nilpotents and unipotents.  Up to conjugacy, over algebraically closed fields parabolic subgroups correspond to partitions of $m$ and every nilpotent orbit is the Richardson orbit of such a parabolic.  Let $\rhobar(\tau) - 1_m =: \overline{N}$ correspond to the partition $\sigma = n_1 + n_2 + \ldots + n_r$.  By Example~\ref{ex:glnrep}, the lift $N_\sigma$ of $\Nbar$ is conjugate to a block nilpotent matrix with blocks of size $n_1, n_2, \ldots ,n_r$.  The points $N \in \Nil_{\Nbar}(R)$ are the $\Ghat(R)$-conjugates of $N_\sigma$.  It is clear (since $p>m$) that if $\rho(\tau) \in \Nil_{\Nbar}(R)$ then
 \begin{align}
  \ker (\rho(\tau) - 1_m)^i \tensor{R} k \to \ker(\rhobar(\tau)-1_m)^i
 \end{align}
is an isomorphism for all $i$.  Conversely, repeated applications of \cite[Lemma 2.4.15]{cht08} show that any $\rho(\tau)$ satisfying this collection of isomorphism conditions is $\Ghat(R)$-conjugate to $N_\sigma$.  Thus the minimally ramified deformation condition for $\GL_m$ defined in \cite{cht08} agrees with our definition.  Note that the identification $X \mapsto  1_m +X$  does not satisfy $q X \to (1+X)^q$, so it will not work in our argument.  The proof of \cite[Lemma 2.4.19]{cht08} uses a different method for which this non-homomorphic identification suffices.
\end{example}

The functor of minimally ramified lifts is pro-representable by a ring $R^{\mr \square}_\rhobar$  as it suffices to specify images of $\tau$ and $\phi$ subject to constraints that $\rho(\tau) = \exp(N)$ and $\rho(\phi) \rho(\tau) \rho(\phi)^{-1} = \rho(\tau)^q$. 

\begin{prop} \label{prop:smoothnessmrtame}
Under assumptions \ref{assumption1}-\ref{assumption4}, the lifting ring $R^{\mr \square}_{\rhobar}$ is formally smooth over $\O$ of relative dimension $\dim G_k$.
\end{prop}

\begin{proof}
Let $\Phibar = \rhobar(\phi) \in G(k)$ and let $\Ghat_\Phibar$ be the formal completion of $G$ at $\Phibar$.  Using the relation $$\rhobar(\phi) \rhobar(\tau) \rhobar(\phi)^{-1} = \rhobar(\tau)^q,$$ we deduce that $\Phibar \, \overline{N} \, \Phibar^{-1} = q \Nbar$.  Therefore we study the functor $M_{\Nbar}$ on $\Chat_\O$ defined by
\[
 M_{\Nbar}(R) = \{ (\Phi,N) : N \in \Nil_{\Nbar} (R),\, \Phi \in \Ghat_\Phibar(R), \, \Phi N \Phi^{-1} = q N \} \subset \Nil_{\Nbar}(R) \times \Ghat_\Phibar(R).
\]
Any such lift $(\Phi,N)$ to a coefficient ring $R$ determines a homomorphism $T_q \to G(R)$ lifting $\rhobar$ via $\phi \mapsto \Phi$ and $\tau \mapsto \exp(N)$: it is continuous because $\exp(\Nbar)$ is unipotent.  We will analyze $M_{\Nbar}$ through the composition
\[
 M_{\Nbar} \to \Nil_{\Nbar} \to \Spf \O.
\]
First, observe that $M_{\Nbar} \to \Nil_{\Nbar}$ is relatively representable as ``$\Phi N = q N \Phi$'' is a formal closed condition on points $\Phi$ of $(\Ghat_{\overline{\Phi}})_R$ for each $N \in \Nil_{\Nbar}(R)$.

From Lemma~\ref{lem:nilpliftable}, we know that $\Nil_{\Nbar}$ is formally smooth over $\O$, and the universal nilpotent is $g N_\sigma g^{-1}$ for some $g \in \Ghat(\Nil_{\Nbar})$.  To check formal smoothness of the map  $M_{\Nbar} \to \Nil_{\Nbar}$, it therefore suffices to check the formal smoothness of the fiber of $M_{\Nbar}$ over the $\O$-point $N_\sigma$ of $\Nil_{\Nbar}$.

We have written down $\Phi_\sigma \in G(\O)$ satisfying $\Phi_\sigma N_\sigma \Phi^{-1}_\sigma = q N_\sigma$ in Remark~\ref{rmk:phi}.  Observe that $ \Phibar \, \Phibar_\sigma^{-1} \in Z_G(N_\sigma)(k)$.  By smoothness, we may lift $\Phibar \, \Phibar_\sigma^{-1}$ to an element $s \in Z_G(N_\sigma)(\O)$.  Then $s \Phi_\sigma$ reduces to $\Phibar$ and satisfies $(s \Phi_\sigma) N_\sigma (s \Phi_\sigma)^{-1} = q N_\sigma$, so the fiber of $M_{\Nbar}$ over $N_\sigma$ has an $\O$-point.  
The relative dimension of the formally smooth $\Nil_{\Nbar}$ is $\dim G_k - \dim Z_{G_k}(\Nbar)$ by Lemma~\ref{lem:nilpliftable}, and $M_{\Nbar} \to \Nil_{\Nbar}$ is a $\widehat{Z_G(N_\sigma)}$-torsor since it has an $\O$-point over $N_\sigma$.  As $Z_G(N_\sigma)$ is smooth it follows that $M_{\Nbar}$ is formally smooth over $\Spf \O$ of relative dimension $\dim G_k$.
\end{proof}

\begin{example}
This recovers \cite[Lemma 2.4.19]{cht08} in the case $G = \GL_m$.  
\end{example}

Let $S$ be the (torus) quotient of $G$ by its derived group $G'$, and $\mu : G \to S$ the quotient map.  For use later, we now study a variant where we fix a lift $\nu : T_q \to S(\O)$ of $\mu \circ \rhobar : T_q \to S(k)$:

\begin{cor} \label{cor:mrtamefixed}
Under assumptions \ref{assumption1} to \ref{assumption4}, the deformation condition of minimally ramified lifts $\rho : T_q \to G(R)$ satisfying $\mu \circ \rho = \nu$ is a liftable deformation condition.  The lifting ring $R_\rhobar^{\mr, \nu,  \square}$ is of relative dimension $\dim G_k - 1$, and the tangent space to the deformation functor has dimension $\dim_k H^0(T_q, \adrho)-1$.
\end{cor}

\begin{proof}
The last claim about the dimension of the tangent space to the deformation functor follows from the claim about the dimension of the lifting ring and Remark~\ref{rem:framedvdef}.  

The quotient torus $S = G / G'$ is split of rank $1$, so the subscheme $R_\rhobar^{\mr, \nu, \square} \subset R_\rhobar^{\mr \square}$ is the vanishing of locus of a single function.  As $R_\rhobar^{\mr \square}$ is formally smooth over $\O$ with relative dimension $\dim G_k$, it suffices to check that the tangent space of $R_\rhobar^{\mr, \nu, \square}$ over $k$ (in the sense of \cite[\S15]{mazur95}) is a proper subspace of the tangent space of $R_\rhobar^{\mr \square}$: this will establish formal smoothness and the dimension claim.

Let $Z$ be the maximal central torus of $G$.  On the level of Lie algebras, we know that $\Lie G$ splits over $\O$ as a direct sum of $\Lie G'$ and $\Lie S \simeq \Lie Z$ as $p$ is very good for $G$.  We can modify a lift $\rho_0$ over $R = k[\epsilon]/(\epsilon^2)$ by multiplying against an unramified non-trivial character $T_q \to Z(R)$ with trivial reduction, changing $\mu \circ \rho_0$.  Thus the tangent space of $R_\rhobar^{\mr, \nu, \square}$ is a proper subspace of that of $R_\rhobar^{\mr \square}$.
\end{proof}

\subsection{Deformation Conditions Based on Parabolic Subgroups} \label{sec:parabolicdefcondition}
The use of nilpotent orbits is not the only approach to defining a deformation condition at ramified places not above $p$.  As discussed in the introduction, the method used to prove \cite[Lemma 2.4.19]{cht08} suggests a generalization from $\GL_n$ to other groups $G$ based on deformations lying in certain parabolic subgroups of $G$.  This deformation condition is not smooth for algebraic groups beyond $\GL_n$, so it does not work in Ramakrishna's method.  In this section we give a conceptual explanation for this phenomenon. 

Let $P \subset G$ be a parabolic $\O$-subgroup.  The Richardson orbit for $P_k$ intersects $(\Lie R_u P)_k$ in a dense open set which is a single geometric orbit under $P_k$.  Suppose that $\rhobar(\tau)$ is the exponential of a $k$-point $\overline{N}$ in the Richardson orbit, and consider deformations $\rho: T_q \to G(\O)$ of $\rhobar$ ramified with respect to $P$ in the sense that $\rho(\tau) \in P$ (compare with Definition~\ref{defn:ramifiedwrtp}).  This requires specifying a lift of $\overline{N}$ that lies in $\Lie P$.  One could hope that such lifts would automatically be $G(\O)$-conjugate to the fixed lift $N_\sigma$ defined in Proposition~\ref{prop:symorthreps}, reminiscent of the definition we gave for $\Nil_{\Nbar}$, a situation in which the associated deformation (or lifting) ring is smooth.

We now show that often smoothness fails if $\overline{N}$ does \emph{not} lie in the Richardson orbit of $P_k$.  Lifts of $\overline{N}$  can ``change nilpotent type'' yet still lie in a parabolic lifting $P_k$, such as the example of the standard Borel subgroup in $\GL_3$ with
\[
 N = \begin{pmatrix}
      0 & 1 & 0 \\
      0 & 0 & p \\
      0 & 0 & 0
     \end{pmatrix} \quad \text{lifting } \quad \overline{N} =
     \begin{pmatrix}
      0 & 1 & 0 \\
      0 & 0 & 0 \\
      0 & 0 & 0
     \end{pmatrix}
\]
In particular, we easily obtain non-pure nilpotents.  This is very bad: the nilpotent orbits over a field are smooth but the nilpotent cone is not smooth, so the deformation problem of deforming with respect to $P$ should not be smooth because ``it sees multiple orbits''.   Furthermore, even if we could lift $\rhobar(\tau)$ appropriately, there would still be problems lifting $\rhobar(\phi)$ because the centralizer of a non-pure nilpotent is not smooth over $\O$ (the special and generic fiber typically have different dimensions).  So it is crucial to choose a parabolic such that $\overline{N}$ lies in the Richardson orbit of $P_k$.  

For $\GL_n$, \emph{all} nilpotent orbits are Richardson orbits. This is not true in general.  In particular, we should not expect the deformation condition of being ramified with respect to a parabolic to be liftable.  Example~\ref{ex:unliftable} illustrates this phenomenon for $\GSp_4$, which we now revisit in a more conceptual manner.

\begin{example}
Take $G = \GSp_4$. Parabolic subgroups correspond to isotropic flags.  Up to conjugacy, these subgroups are stabilizers of the flags
\begin{align*}
 0 \subset \Span(v_1) \subset \Span(v_1,v_2) &\subset \Span(v_1,v_2,v_3) \subset k^4, 0 \subset k^4 \\
 0 \subset \Span(v_1) \subset \Span(v_1,v_2,v_3) \subset &k^4, \quad  0 \subset \Span(v_1,v_2) \subset k^4
\end{align*}
where $\{v_1,v_2,v_2,v_4\}$ is a basis of $k^4$.
Their Richardson orbits correspond to the nilpotent orbits indexed respectively by the partitions $1+1+1+1$, $4$, $4$, and $2+2$.  In particular, the same nilpotent orbit is associated with two flags, and the nilpotent orbit corresponding to the partition $2+1+1$ does not appear.  This corresponds to a nilpotent orbit that is not a Richardson orbit; for the representation in Example~\ref{ex:unliftable}, $\log(\rhobar(\tau))$ is in this nilpotent orbit.
\end{example}

\section{Minimally Ramified Deformations in General} \label{sec:mr}
We continue the notation of the previous section.  We have defined the minimally ramified deformation condition for representations factoring through the quotient $T_q = \ZZhat \ltimes \ZZ_p$ of the tame Galois group $\Gamma_L^\tame$ at a place away from $p$.   In this section, we will adapt the matrix-theoretic methods in \cite[\S 2.4.4]{cht08}, making use of more conceptual module-theoretic arguments, to define the minimally ramified deformation condition for any representation when $G = \GSp_{m}$ or $G= \GO_m$.  (Minor variants of this method work for $\Sp_{m}$ and $\SO_m$, and the original method of \cite[\S 2.4.4]{cht08} works for $\GL_m$.)  We naturally embed $G$ into $\GL(M)$ for a free $\O$-module $M$ of rank $m$, and let $V$ denote the reduction of $M$, a vector space over the residue field $k$.

We consider a representation $\rhobar : \Gamma_L \to G(k) \subset \GL(V)(k)$ which may be wildly ramified (with $L$ an $\ell$-adic field for $\ell \neq p$).  We will define a deformation condition for $\rhobar$ in terms of the minimally ramified deformation condition for certain associated tamely ramified representations, after possibly extending $\O$.  In \S\ref{sec:decomposing}, we analyze $\rhobar$ as being built out of two pieces of data: representations of a closed normal subgroup $\Lambda_L$ of $\Gamma_L$ whose pro-order is prime to $p$, and tamely ramified representations of $\Gamma_L/\Lambda_L$.  The representation theory of $\Lambda_L$ is manageable since its pro-order is prime to $p$, and representations of $\Gamma_L/\Lambda_L$ can be understood using the results of the previous section.

\subsection{Decomposing Representations} \label{sec:decomposing}

We begin with a few preliminaries about representations over rings.  Let $\Lambda'$ be a profinite group and $R$ be an Artinian coefficient ring with residue field $k$.  If $\Lambda'$ has pro-order prime to $p$, the representation theory of $\Lambda'$ over $k$ is nice: every finite-dimensional continuous representation is a direct sum of irreducibles, and every such representation is projective over $k[\Lambda]$ for any finite discrete quotient $\Lambda$ of $\Lambda'$ through which the representation factors.  We are also interested in corresponding statements over an Artinian coefficient ring $R$.

\begin{fact} \label{fact:ringrep}
Let $R$ be an Artinian coefficient ring with residue field $k$.  Suppose the pro-order of $\Lambda'$ is prime to $p$.
Let $P$ and $P'$ be $R[\Lambda']$-modules that are finitely generated over $R$ with continuous action of $\Lambda'$, and $F$ be a $k[\Lambda']$-module that is finite dimensional over $k$ with continuous action of $\Lambda'$.  Let $\Lambda$ be a finite discrete quotient of $\Lambda'$ through which the $\Lambda'$-actions on $P$, $P'$, and $F$ factor.
\begin{enumerate}
 \item   If $P$ is free as an $R$-module, it is projective as a $R[\Lambda]$-module. \label{ringrep1}
 
 \item  If $P$ and $P'$ are projective over $R[\Lambda]$, they are isomorphic if and only if $\overline{P}$ and $\overline{P'}$ are isomorphic. \label{ringrep2}
 
 \item  There exists a projective $R[\Lambda]$-module (unique up to isomorphism) whose reduction is $F$. \label{ringrep3}
\end{enumerate}
\end{fact}

These statements are special cases of results in \cite[\S 14.4]{serre77}.  We now record two lemmas which do not need the assumption that the pro-order of $\Lambda'$ is prime to $p$.  Here and elsewhere, we use $\Hom_{\Lambda'}$ to denote homomorphisms as representations of $\Lambda'$ (equivalently as $R[\Lambda']$-modules).

\begin{lem} \label{lem:hom}
 Let $P$ and $P'$ be $R[\Lambda']$-modules, finitely generated over $R$ with continuous action of $\Lambda'$ factoring through a finite discrete quotient $\Lambda$ of $\Lambda'$.  Assume $P$ and $P'$ are $R[\Lambda]$-projective.  The natural map gives an isomorphism \[\Hom_{\Lambda'}(P,P') \tensor{R} k \to \Hom_{\Lambda'}(\overline{P},\overline{P'}).\]
\end{lem}

\begin{proof} 
We may replace $\Hom_{\Lambda'}$ with $\Hom_{\Lambda}$.  Note that $\m P' = \m \tensor{R} P'$, so $\Hom_{\Lambda}(P, \m P') = \Hom_{\Lambda}(P,P') \tensor{R} \m$ as $P$ and $P'$ are $R[\Lambda]$-projective.  Then apply $\Hom_{\Lambda}(P,-)$ to the exact sequence $0 \to \m P' \to P' \to P'/\m P' \to 0$.
\end{proof}

\begin{lem} \label{lem:zero}
Let $\Lambda$ be a finite group and let $M$ and $M'$ be finite $R[\Lambda]$-modules whose reductions $\Mbar$ and $\Mbar'$ are non-isomorphic irreducible $k[\Lambda]$-modules.  Then $\Hom_{R[\Lambda]}(M,M') =0$.
\end{lem}

\begin{proof}
 Filter $M'$ by the composition series $\{ \m^i M'\}$, and consider the surjection
 \[
  \m^i / \m^{i+1} \otimes \overline{M}' \onto \m^i M' / \m^{i+1} M'.
 \]
Now $\Lambda$ acts on $\m^i / \m^{i+1} \otimes M'$ via its action on $\overline{M}'$, so as a $k[\Lambda]$-module $\m^i M' / \m^{i+1} M'$ is isomorphic to a direct sum of copies of $\Mbar'$.  Thus
\[
 \Hom_{R[\Lambda]}( M , \m^i M' / \m^{i+1} M') = \Hom_{k[\Lambda]}(\Mbar , \m^i M' / \m^{i+1} M') =0
\]
as $\Mbar$ and $\Mbar'$ are non-isomorphic $k[\Lambda]$-modules.

By descending induction on $i$, we shall show that
\[
 \Hom_{R[\Lambda]} (M, \m^i M') =0.
\]
For large $i$, $\m^i M' =0$.  
Consider the exact sequence
\[
 0 \to \m^{i+1} M' \to \m^i M' \to  \m^i M' / \m^{i+1} M' \to 0.
\]
Applying $\Hom_{R[\Lambda]}(M,-)$, we obtain a left exact sequence
\[
 0 \to \Hom_{R[\Lambda]}(M,\m^{i+1} M' ) \to \Hom_{R[\Lambda]}(M,\m^{i} M' ) \to \Hom_{R[\Lambda]}(M,\m^i M' / \m^{i+1} M')
\]
The left term is $0$ by induction, and the right term is $0$ by the above calculation.  This completes the induction.
\end{proof}

Given $\rhobar : \Gamma_L \to G(k) \subset \GL(V)(k)$ and a lift $\rho : \Gamma_L \to G(R) \subset \GL(M)(R)$ for some $R \in \C_\O$, we now turn to decomposing the $R[\Gamma_L]$-module $M$.
Let $I_L \subset \Gamma_L$ be the inertia group, and pick a surjection $I_L \to \ZZ_p$.  Define $\Lambda_L$ to be the kernel of this surjection (normal in $\Gamma_L$).    
This is a pro-finite group with pro-order prime to $p$, and is independent of the choice of surjection.  Define the quotient $$T_L := \Gamma_L/\Lambda_L,$$ which is a quotient of the tamely ramified Galois group $\Gamma_L^\tame$ and of the form $T_q = \ZZhat \ltimes \ZZ_p$ as in \S\ref{sec:mrtame}.
We wish to compatibly decompose $V$ and $M$ as $\Lambda_L$-modules and then understand the action of $\Gamma_L$ on the decomposition.

We first make a finite extension of $k$ (and of $\O$) so that all of the (finitely many) irreducible representations of $\Lambda_L$ over $k$ occurring in $V$ are absolutely irreducible over $k$.  

Because $\Lambda_L$ has order prime to $p$, $\Res^{\Gamma_L}_{\Lambda_L}(V)$ is completely reducible and we can write
\[
 \Res^{\Gamma_L}_{\Lambda_L}(V) = \bigoplus_{\tau} V_\tau 
\]
where $\tau$ runs through the set of isomorphism classes of irreducible representations of $\Lambda_L$ over $k$ occurring in $V$, and each $V_\tau$ is the $\tau$-isotypic component.  We will obtain an analogous decomposition for $M$.  

Let $\Gamma$ be a finite discrete quotient of $\Gamma_L$ through which $\rho$ factors, and let $\Lambda$ be the image of $\Lambda_L$ in $\Gamma$.
Using Fact~\ref{fact:ringrep}\eqref{ringrep3} we can lift $\tau$ to a projective $R[\Lambda]$-module $\tildetau$ unique up to isomorphism.  We will eventually want this lift to have additional properties (see \S \ref{sec:extendingtau}), but this is not yet necessary.  We set $W_\tau := \Hom_{\Lambda_L}(\tildetau,M)$ and consider the natural morphism
\[
 \bigoplus_\tau \tildetau \tensor{R} W_\tau \to M.
\]
Note that $M$ is $R[\Lambda]$-projective by Fact~\ref{fact:ringrep}\eqref{ringrep1}.

\begin{lem} \label{lem:decomposition}
 This map is an isomorphism of $R[\Lambda_L]$-modules.
\end{lem}

\begin{proof}
It suffices to check it is an isomorphism of $R[\Lambda]$-modules.  When $R=k$, $\End_{\Lambda}(\tau)=k$ as we extended $k$ so that all of the irreducible representations of $\Lambda$ over $k$ occurring inside $V$ are absolutely irreducible.  Splitting up $V$ as a direct sum of irreducibles, we obtain the desired isomorphism.

In the general case, the map is an isomorphism after reducing modulo $\m$ (use Lemma~\ref{lem:hom}).  Thus by Nakayama's lemma it is surjective.  Since $M$ is $R$-projective, the formation of the kernel commutes with reduction modulo $\m$.  Thus, again using Nakayama's lemma the kernel is zero.
\end{proof}

We define $M_\tau$ to be the image of $\tildetau \tensor{R} \Hom_{\Lambda_L}(\tildetau,M)$ in $M$.  It is the largest $R[\Lambda_L]$-direct summand whose reduction is a direct sum of copies of $\tau$. 

We next seek to understand the action of $\Gamma_L$ on this canonical decomposition of $M$.  For $g \in \Gamma_L$, consider the $R[\Lambda_L]$-module $g M_\tau$: it is a direct summand of $M$ over $R$ whose reduction is a direct sum of copies of the representation $\tau^g$ defined by $\tau^g(h) = \tau(g^{-1} h g)$ for $h \in \Lambda_L$.  Thus we see that $g M_\tau = M_{\tau^g}$ inside $M$, and $\Gamma_L$ permutes the $M_\tau$'s.  The orbits corresponds to sets of conjugate representations.

Consider the stabilizer of $V_\tau$:
\[
 \Gamma_{L,\tau} = \{ g \in \Gamma_L : g V_\tau = V_\tau \text{ inside } V \} = \{g \in \Gamma_L : \tau^g \simeq \tau\} \subset \Gamma_L
\]
with corresponding image
\[
 \Gamma_\tau = \{ g \in \Gamma : g V_\tau = V_\tau \text{ inside } V\} = \{g \in \Gamma : \tau^g \simeq \tau\} \subset \Gamma.
\]
Then the $k$-span of the $\Gamma_L$-orbit of $V_\tau$ is exactly the representation $\Ind_{\Gamma_{L,\tau}}^{\Gamma_L} V_\tau = \Ind^{\Gamma}_{\Gamma_\tau} V_\tau$.
Letting $[\tau]$ denote the set of $R[\Lambda]$-isomorphism classes of $\Lambda$-representations $\Gamma$-conjugate to $\tau$, by taking into account the action of $\Gamma_\tau$ the isomorphism in Lemma~\ref{lem:decomposition} becomes an isomorphism of $R[\Gamma_L]$-modules
\begin{equation} \label{eq:mdecomposition}
 \bigoplus_{[\tau]} \Ind_{\Gamma_{L,\tau}}^{\Gamma_L} M_\tau \xrightarrow{\sim} M
\end{equation}
using one representative $\tau$ per $\Gamma_L$-conjugacy class $[\tau]$.

For orthogonal or symplectic representations, we will make precise the notion that this decomposition is ``compatible with duality''.  
Denote the similitude character by $\mu$, and let 
 $\overline{\nu} := \mu \circ \rhobar : \Gamma_L \to k^\times$.
Let $N$ be a free $\O$-module of rank $1$ on which $\Gamma_L$ acts by a specified continuous $\O^\times$-valued lift $\nu$ of $\overline{\nu}$, and let $\overline{N}$ be its reduction modulo $\m$.
For an $R$-module $M$, define $M^\vee = \Hom_{R}(M,N_R)$ with the evident $\Gamma_L$-action.  

Now suppose we have chosen $\nu$ so that $\nu = \mu \circ \rho$ (viewed as maps $\Gamma_L \to R^\times$).  The perfect pairing on the $M$ corresponding to $\rho$ gives an isomorphism of $R[\Gamma_L]$-modules $\psi: M \simeq M^\vee$.  In particular, using Lemma~\ref{lem:decomposition} we see that
\[
 \bigoplus_\tau M_\tau = M \overset{\psi} \simeq M^\vee = \bigoplus_\tau (M_\tau)^\vee,
\]
To simplify notation, we will write $M_\tau^\vee$ for $(M_\tau)^\vee$.  Note that the right side is also an isotypic decomposition, with $M_\tau^\vee$ the maximal direct whose direct sum is a direct sum of copies of $\tau^\vee$.    By comparing isotypic pieces, we obtain a natural isomorphism (of $R[\Lambda_L]$-modules)
\[
 \psi_\tau: M_\tau^\vee \simeq M_{\tau^*}
\]
for some irreducible representation $\tau^*$ of $\Lambda_L$ occurring in $V$.  Note that $\tau^* \simeq \tau^\vee$ as $k[\Lambda_L]$-modules, but that $\tau^*$ is not necessarily the dual of $\tau$ as $k[\Gamma_L]$-modules.  There are three cases:
\begin{itemize}
 \item {\bf Case 1:}  $\tau$ is not $\Gamma_L$-conjugate to $\tau^*$;
 \item {\bf Case 2:} $\tau$ is isomorphic to $\tau^*$ as $\Gamma_L$-modules;
 \item {\bf Case 3:} $\tau$ is $\Gamma_L$-conjugate to $\tau^*$ but not isomorphic.
\end{itemize}

In the second case, we claim that the isomorphism of $k[\Lambda_L]$-modules $\imath : \tau \simeq \tau^\vee$ gives a sign-symmetric (for some fixed sign $\epsilon_\tau$) perfect pairing on $\tau$.  Note that $\overline{W}_\tau = \Hom_{\Lambda}(\tau,V)$ by Lemma~\ref{lem:hom}, and that
$$\overline{W}_\tau = \Hom_{\Lambda}(\tau,V) \overset{\overline{\psi}} \simeq \Hom_{\Lambda}(\tau,V^\vee) \overset{\imath} \simeq \Hom_{\Lambda}(\tau^\vee, V^\vee) \simeq \overline{W}_\tau^\vee.$$  Denote this isomorphism by $\varphi_\tau$: it defines a pairing $\langle , \rangle_{\overline{W}_\tau}$ on $\overline{W}_\tau$ via
\[
 \langle w_1, w_2 \rangle_{\overline{W}_\tau} := \varphi_\tau(w_1)(w_2).
\]
We can also 
define $\langle v_1,v_2 \rangle_{\tau} := \imath(v_1)(v_2)$ for $v_1, v_2 \in \tau$.  We have a commutative diagram of isomorphisms
\[
\xymatrixcolsep{5pc}
 \xymatrix{
  \tau \otimes \overline{W}_\tau \ar^{\id \otimes \varphi_\tau} [r] \ar^{\on{eval}}[d] &  \tau \otimes \overline{W}_\tau^\vee \ar^{\imath \otimes \id} [r] & \tau^\vee \otimes \overline{W}_\tau^\vee \ar^{\on{eval}^\vee}[d] \\
  V_\tau \ar^{\overline{\psi_\tau}} [rr]&  & V_\tau^\vee
 }
\]
The commutativity says that for elementary tensors 
$m_i = v_i \otimes w_i \in V_\tau = \tau \otimes \overline{W}_\tau$ we have
\begin{align} \label{eq:vpairingseq}
\begin{split}
 \langle m_1,m_2 \rangle_M = \overline{\psi}(m_1)(m_2)& = \left( \imath(v_1) \otimes \varphi_\tau(w_1) \right) ( v_2 \otimes w_2) \\
  &= \imath(v_1)(v_2) \cdot \varphi_\tau(w_1)(w_2) = \langle v_1, v_2 \rangle_{\tau} \langle w_1, w_2 \rangle_{\overline{W}_\tau}.
\end{split}
\end{align}

Remember that the pairing on $V$ is $\epsilon$-symmetric.

\begin{lem} \label{lem:signsymmetricpairing}
The pairing $\langle , \rangle_{\tau}$ is a sign-symmetric (for fixed sign $\epsilon_\tau$).
\end{lem}

\begin{proof}
Suppose there exists $v \in \tau$ such that $\imath(v) ( v) \neq 0$.  For $w_1, w_2 \in \overline{W}_\tau$, \eqref{eq:vpairingseq} gives
\[
\imath(v) (v) \varphi_\tau(w_1)(w_2) =  \langle v \otimes w_1, v \otimes w_2 \rangle_V  = \epsilon  \langle v \otimes w_2, v \otimes w_1 \rangle_V = \epsilon \imath(v) (v) \varphi_\tau(w_2)(w_1).
\]
Canceling $\imath( v) (v)$, we conclude that $\langle w_1,w_2 \rangle_{\overline{W}_\tau} = \epsilon \langle w_2, w_1\rangle_{\overline{W}_\tau}$.  Using \eqref{eq:vpairingseq}, we conclude that
\[
 \epsilon \imath(v_2)(v_1) \cdot \varphi_\tau(w_2)(w_1) = \epsilon \langle m_2, m_1 \rangle_V = \langle m_1,m_2 \rangle_V = \imath(v_1)(v_2) \cdot \varphi_\tau(w_1)(w_2) = \epsilon \imath(v_1)(v_2) \cdot \varphi_\tau(w_2)(w_1).
\]
Choosing $w_1$ and $w_2$ with $\langle w_2, w_1 \rangle_{\overline{W}_\tau} \neq 0$ (possible as $\langle , \rangle_V$ is perfect), we then conclude that $\langle v_1,v_2 \rangle_{\tildetau} = \langle v_2, v_1 \rangle_{\tildetau}$.

Otherwise $\imath(v) ( v) =0$ for all $v \in \tau$, in which case $\langle, \rangle_{\tildetau}$ is alternating.
\end{proof}

In \S\ref{sec:extendingtau} we will see that the action of $\Lambda_L$ on the module underlying $\tildetau$ can be extended to an action of $\Gamma_{L,\tau}$ factoring through $\Gamma_\tau$.  Therefore, $W_\tau = \Hom_{\Lambda_L}(\tildetau,M)$ is naturally a representation of $T_{L,\tau} := \Gamma_{L,\tau}/\Lambda_L$, and of $T_\tau := \Gamma_\tau/\Lambda$ (a finite quotient of $T_{L,\tau}$).  In \S\ref{sec:mrcondition}, we will use the minimally ramified deformation condition of \S\ref{sec:mrtame} to specify which deformations $W_\tau$ are allowed.  Together with the decomposition~\eqref{eq:mdecomposition}
\[
  \bigoplus_{[\tau]} \Ind^{\Gamma_L}_{\Gamma_{L,\tau}} \left(\tildetau \otimes W_\tau \right) \to M
\]
this defines a deformation condition for $\rhobar$.  Some care is needed to ensure compatibility of the lifts with the pairing on $M$, which will require breaking into cases in the next sections based on the relationship between $\tau$ and $\tau^*$.  

\subsection{Extension of Representations} \label{sec:extendingtau}

We continue the notation of the previous section, where $\tau$ is an absolutely irreducible representation of $\Lambda_L$ over $k$.  We need to lift this to a representation over $\O$ and extend it to a representation of $\Gamma_{L,\tau}$.  We will have to do something extra for the representation to be compatible with a pairing, depending on how $\tau$ and $\tau^*$ are related.

In {\bf Case 1}, we ignore the pairing.  Lemma 2.4.11 of \cite{cht08} lets us pick a $\O[\Gamma_{L,\tau}]$-module $\tildetau$ that is a free $\O$-module and reduces to $\tau$.  In this case, $\tildetau^\vee$ is a free $\O$-module reducing to $\tau^*$.

In {\bf Case 2}, from Lemma~\ref{lem:signsymmetricpairing} it follows that $\tau$ is a symplectic or orthogonal representation.  We will adapt the $\GL_n$-technique of \cite{cht08} to produce a symplectic or orthogonal extension $\tildetau$.
Letting $n = \dim \tau$, the representation $\tau$ gives a homomorphism $\tau : \Lambda_L \to G(k)$ where $G$ is $\GSp_n$ or $\GO_n$.

First, we claim that there is a continuous lift $\tildetau : \Lambda_L \to G(W(k))$: without the pairing, this would be Fact~\ref{fact:ringrep}\eqref{ringrep3}.  To also take into account the pairing, consider deformation theory for the residual representation $\tau$.  This is a smooth deformation condition as $H^2(\Lambda_L,\on{ad} \tau)=0$: $\Lambda_L$ has pro-order prime to $p$ and $\on{ad} \tau$ has order a power of $p$.  Therefore the desired lift exists.  It is unique (up to conjugation by an element of $\widehat{G}(\O)$) because the tangent space is zero dimensional as $H^1(\Lambda_L,\on{ad}\tau)=0$.  By considering representations of the group $\Lambda_L/\ker(\tau)$,  we may and do assume that $\ker(\tildetau) = \ker(\tau)$ as subgroups of $\Lambda_L$.

\begin{remark}
For $g \in \Gamma_{L,\tau}$, the $k[\Lambda_L]$-modules $\tau^g \simeq \tau$ are isomorphic.  By uniqueness of the lift, this means that there is $A \in G(\O)$ such that $\tildetau^g(\gamma) = A \tildetau(\gamma) A^{-1}$ for all $\gamma \in \Gamma_{L}$.  Furthermore, $\Gamma_{L,\tau} = \{ g \in \Gamma_L : \tildetau^g \simeq  \tildetau \}$. \label{remark:gtau}
\end{remark}

We will now show how to continuously extend $\tildetau$ to $\Gamma_{L,\tau}$.  
The first step in constructing the extension is to understand the structure of $\Gamma_{L,\tau}$ and $I_L \cap \Gamma_{L,\tau}$, where $I_L$ is the inertia group.

Recall that $T_L = \Gamma_L/\Lambda_L$ is the semi-direct product of $\ZZhat$ and $\ZZ_p$, where $\ZZhat$ is generated by a lift of Frobenius $\phi$ and $\ZZ_p$ is generated by an element $\sigma$, with $\phi \sigma \phi^{-1} = \sigma^q$ where $q = \ell^a$ is the size of the residue field of $L$.

\begin{fact} \label{fact:topsplit}
The exact sequence
\[
 1 \to \Lambda_L \to \Gamma_L \to T_L \to 1
\]
is topologically split, so $\Gamma_L$ is a semi-direct product.
\end{fact}

\begin{proof}
This is \cite[2.4.10]{cht08}.
\end{proof}

For $T_{L,\tau} := \Gamma_{L,\tau} / \Lambda_L$, this gives a topological splitting of
\[
 1 \to \Lambda_L \to \Gamma_{L,\tau} \to T_{L,\tau} \to 1.
\]
As $\Gamma_{L,\tau}$ is an open subgroup of $\Gamma_L$, we observe that $T_{L,\tau}$ is an open subgroup of $T_L$.  Note that $T_{L,\tau}$ is normal and topologically generated by some powers of $\phi$ and $\sigma$ which will be denoted by $\phi_\tau$ and $\sigma_\tau$ (since any open subgroup of a semidirect product $C \ltimes C'$ for pro-cyclic $C$ and $C'$ is of the form $C_0 \ltimes C'_0$ for open
subgroups $C_0 \subset C$ and $C'_0 \subset C'$).  In particular, using the notation of \S\ref{sec:mrd}
$T_{L,\tau}$ is itself isomorphic to $T_{q'}$ for some $q'$.
The element $\sigma_\tau$ and $\Lambda_L$ together topologically generate $\Gamma_{L,\tau} \cap I_L$.  

Before extending $\tildetau$, we need several technical lemmas.

\begin{lem} \label{lem:centlem}
We have that $\End_{\Lambda_L}(\tildetau) = \O$.
\end{lem}

\begin{proof}
As $\tau$ is absolutely irreducible, $\End_{\Lambda_L}(\tau) = k$.  By Lemma~\ref{lem:hom}, we see that the reduction of $\End_{\Lambda_L}(\tildetau)$ modulo the maximal ideal of $\O$ is $k$, so the map $\O \into \End_{\Lambda_L}(\tildetau)$ is surjective by Nakayama's lemma.
\end{proof}

\begin{lem} \label{lem:pdivide}
 The dimension of $\tau$ is not divisible by $p$.
\end{lem}

\begin{proof}
 As $\tau$ is continuous and $\Lambda_L$ has pro-order prime to $p$, the representation $\tau$ factors through a finite discrete quotient $\Lambda$ of $\Lambda_L$ whose order is prime to $p$.  Such a representation is the reduction of a projective $\O[\Lambda]$-module by Fact~\ref{fact:ringrep}\eqref{ringrep3}.  Inverting $p$, we obtain a representation of $\Lambda$ in characteristic zero that is absolutely irreducible since the ``reduction'' $\tau$ is absolutely irreducible over $k$.  By \cite[\S6.5 Corollary 2]{serre77}, the dimension of this representation (equal to the dimension of $\tau$) divides the order of $\Lambda$.
\end{proof}

We will now extend $\tildetau$ from $\Lambda_L \subset I_L$ to $\Gamma_{L,\tau}$ by defining it on the topological generators $\sigma_\tau$ and $\phi_\tau$.  We say  that such an extension has \emph{tame determinant} if $\det(\tildetau(\sigma_\tau))$ has finite order which is prime to $p$.  Lemmas~\ref{lem:intermediatetau} and \ref{taulifting} adapt \cite[Lemma 2.4.11]{cht08} and fill in some details.  

\begin{lem} \label{lem:intermediatetau}
There is a unique continuous extension $\tildetau : \Gamma_{L,\tau} \cap I_L \to G(\O)$ with tame determinant.
\end{lem}

\begin{proof}  
A continuous extension of $\tildetau$ to $\Gamma_{L,\tau} \cap I_L$ is determined by its value on $\sigma_\tau$.  As $\sigma_\tau \in \Gamma_{L,\tau}$, in light of Remark~\ref{remark:gtau} there is an $A \in G(\O)$ such that for $g \in \Lambda_L$ we have
\[
 \tildetau(\sigma_\tau g \sigma_\tau^{-1} ) = A \tildetau(g) A^{-1}.
\]
We would like to send $\sigma_\tau$ to the element $A$.  For an appropriate modification of $A$ (still lying in $G(\O)$), this will produce a continuous extension with tame determinant.  
As $\sigma_\tau$ is a topological generator for a group isomorphic to $\ZZ_p$, the continuity of the extension with $\sigma_\tau \mapsto A$ is equivalent to some $p$-power of $A$ having trivial reduction.   We wish to show that there is a unique choice of such $A$ that also makes the extension have tame determinant.

We will first show that some power $A^{p^b}$ lies in the centralizer of the image $\tildetau(\Lambda_L)$.  Consider the conjugation action of $\langle \sigma_\tau \rangle$ on $\Lambda_L$.  As $\ker \tildetau = \ker \tau$ is a normal subgroup of $\Gamma_{L,\tau}$ (if $g \in \Gamma_{L,\tau}$ and $\tau(h) = 1$, then $\tau^g(h)$ is conjugate to $\tau(h) = 1$ by Remark~\ref{remark:gtau}) we get an action of $\langle \sigma_\tau \rangle$ on $\Lambda_L / \ker \tau \simeq \tau (\Lambda_L)$.  
  The action is continuous, so there is a power $p^b$ such that for all $g \in \Lambda_L$ we have
\[
 \tau(\sigma_\tau^{p^b} g \sigma_\tau^{-p^b}) = \tau(g).
\]
As $\ker \tildetau = \ker \tau$, we see that
\[
 A^{p^b} \tildetau (g) A^{-p^b} = \tildetau(\sigma_\tau^{p^b} g \sigma_\tau^{-p^b}) = \tildetau(g).
\]
Therefore $A^{p^b}$ lies in the centralizer of $\tildetau(\Lambda_L)$ in $G(\O)$.

By Lemma~\ref{lem:centlem}, this centralizer is isomorphic to $\O^\times$.  We claim that by multiplying $A$ by some unit in $\O$, we can arrange for the continuous extension $\tildetau$ to exist and have tame determinant.  We will use the fact that an element of $\O^\times$ is the product of a $1$-unit and a Teichmuller lift of an element of $k^\times$. 
As $A^{p^b} \in \O^\times$ and the $p$th power map is an automorphism of $k^\times$, we may multiply $A$ by a unit scalar so that $A^{p^b}$ reduces to the identity matrix.  By Lemma~\ref{lem:pdivide}, the dimension $n$  of $\tau$ is prime to $p$ so we may multiply $A$ by a $1$-unit so that $\det(A)$ has finite order prime to $p$.  Sending $\sigma_\tau$ to this particular $A$ gives a continuous extension with tame determinant.

Let's show this extension is unique.  Any extension must send $\sigma_\tau$ to an element of the form $w A$ for $w \in \O^\times$ (the centralizer of the image $\tildetau(\Lambda_L)$).  By continuity, there is a power $p^b$ such that $(wA)^{p^b}$ reduces to the identity.  This means that $w^{p^b}$ reduces to the identity, and hence that $w$ reduces to the identity.  On the other hand, $\det(w A)\det(A)^{-1} = w^n$.  The left side has finite order that is relatively prime to $p$, so $w^n$ does too.  This forces $w^n=1$ since its reduction is $1$.  
But as $n$ is prime to $p$ (Lemma~\ref{lem:pdivide}), the only $n$th roots of unity in $\O^\times$ are Teichmuller lifts.  Therefore $w=1$.
\end{proof}

\begin{lem} \label{taulifting}
 There is a continuous extension $\tildetau : \Gamma_{L,\tau} \to G(\O)$. 
\end{lem}

\begin{proof}
We extend $\tildetau$ in Lemma~\ref{lem:intermediatetau} continuously to $\Gamma_{L,\tau}$ by defining it on $\phi_\tau$.  As $\phi_\tau \in \Gamma_{L,\tau}$, by Remark~\ref{remark:gtau} there is an element $A \in G(\O)$ conjugating $\tildetau : \Lambda_L \to G(\O)$  to $\tildetau^{\phi_\tau} : \Lambda_L \to G(\O)$.  Each has a unique extension to a continuous morphism from $I_L \cap \Gamma_{L,\tau}$ to $G(\O)$ with tame determinant.  Therefore for $g \in I_L \cap \Gamma_{L,\tau}$ we have
 \[
  \tildetau( \phi_\tau g \phi_\tau^{-1}) = A \tildetau(g) A^{-1}
 \]
 since the right side has the same (tame) determinant as $\tildetau$ on $I_L \cap T_\tau$.  We can continuously extend $\tildetau: I_L \cap \Gamma_{L,\tau} \to G(\O)$ by sending $\phi_\tau$ to $A$ since $A$ has reduction with finite order.
\end{proof}

This gives the desired lift and extension of $\tau$ in the case that $\tau \simeq \tau^*$.

In {\bf Case 3}, $\tau$ is conjugate to $\tau^*$ but not isomorphic.  The argument follows the same structure as the previous case, but we make a few modifications to treat $\tau \oplus \tau^*$ together. In particular, we can pick a copy of the $k[\Lambda_L]$-module $\tau$ inside $V$ and a copy of $\tau^* \simeq \tau^\vee$ inside $V$ such that the pairing restricted to $\tau \oplus \tau^*$ is perfect.  It is sign-symmetric with sign $\epsilon_{\tau \oplus \tau^*}$.

Define $\Gamma_{L,\tau \oplus \tau^*} = \{ g \in \Gamma_L : (\tau \oplus \tau^*)^g \simeq \tau \oplus \tau^*\}$.  It contains $\Gamma_{L,\tau}$ with index $2$, as conjugation either preserves $\tau$ and $\tau^*$ or swaps them.  
Arguing as in the paragraph after Fact~\ref{fact:topsplit}, we obtain a split exact sequence
\[
 0 \to \Lambda_L \to \Gamma_{L,\tau \oplus \tau^*} \to T_{L,\tau \oplus \tau^*} \to 1
\]
where $T_{L,\tau \oplus \tau^*}$ is an open normal subgroup of $T_L$ topologically generated by some powers of $\phi$ and $\sigma$ which we denote by $\phi_{\tau\oplus \tau^*}$ and $\sigma_{\tau \oplus \tau^*}$.  We may arrange that either
\begin{itemize}
 \item {\bf Case 3a:} $\phi_{\tau\oplus \tau^*}^2 = \phi_\tau$ and $\sigma_{\tau \oplus \tau^*} = \sigma_{\tau}$ or
 \item {\bf Case 3b:} $\phi_{\tau\oplus \tau^*} = \phi_\tau $ and $ \sigma_{\tau \oplus \tau^*}^2 = \sigma_{\tau}$ .
\end{itemize}

In {\bf Case 3a},  we begin by lifting $\tau$ to $\O$ as a representation of $\Lambda_L$: as before, we do this using the fact that the pro-order of $\Lambda_L$ is prime to $p$, and obtain a lift $\tildetau$ unique up to isomorphism.  We extend $\tildetau$ to be a representation of $\Gamma_{L,\tau} \cap I_L$ by defining it on $\sigma_{\tau}$ using the $\GL_n$-version of Lemma~\ref{lem:intermediatetau}, \cite[Lemma 2.4.11]{cht08}.  There it is shown all such extensions are unique up to equivalence.  In particular, $\tildetau$ and $(\tildetau^{\phi_{\tau \oplus \tau^*}})^\vee$ are isomorphic $\O[\Gamma_{L,\tau} \cap I_L]$-modules.  We can use this to define a sign-symmetric perfect pairing on $\tildetau \oplus \tildetau^{\phi_{\tau \oplus \tau^*}}$ that is compatible with the action of $\Gamma_{L,\tau} \cap I_L$ and $\phi_{\tau \oplus \tau^*}$, hence of $\Gamma_{L,\tau \oplus \tau^*}$.

In {\bf Case 3b}, as $\tau^\vee$ and $\tau^{\sigma_{\tau \oplus \tau^*}}$ are isomorphic $k[\Lambda_L]$-modules it follows that $\tildetau^\vee$ and $\tildetau^{\sigma_{\tau \oplus \tau^*}}$ are isomorphic $\O[\Lambda_L]$-modules.  In particular, this isomorphism gives a natural way to define a sign-symmetric perfect pairing on $M=\tildetau \oplus \tildetau^{\sigma_{\tau \oplus \tau^*}}$ lifting the residual one.  This pairing is compatible with the action of $\Gamma_{L,\tau \oplus \tau^*} \cap I_L$ (which is generated by $\Lambda_L$ and $\sigma_{\tau \oplus \tau^*}$).  Finally, we claim that $M$ and $M^{\phi_\tau}$ are isomorphic.  As $\phi_\tau \in \Gamma_{L,\tau}$ preserves $\tau$, acting by $\phi_\tau$ gives an isomorphism $\overline{\psi} : \overline{M} \simeq \overline{M}^{\phi_\tau}$ of $k[\Lambda_L]$-modules such that $\overline{\psi}(\tau) = \tau^{\phi_\tau}$.  By uniqueness of the lift of $\tau$ as a $\O[\Lambda_L]$-module, we obtain an isomorphism $\psi_\tau$ of $\tildetau$ and $\tildetau^{\phi_\tau}$ and hence an isomorphism $\psi : M \simeq M^{\phi_\tau}$ via the identifications
\[
 \tildetau^{\sigma_{\tau \oplus \tau^*}} \simeq \tildetau^\vee \overset{ \psi_\tau^{-1}} \simeq (\tildetau^{\phi_\tau^{-1}})^\vee \simeq (\tildetau^\vee)^{\phi_\tau} \simeq (\tildetau^{\sigma_{\tau \oplus \tau^*}})^{\phi_\tau} .
\]
This isomorphism is compatible with the pairing.  The key observation is that for $m \in \tildetau$ and $f \in \tildetau^\vee \simeq \tildetau^{\sigma_{\tau \oplus \tau^*}}$, we have that
\[
 \langle \psi(m), \psi(f) \rangle_M = \langle \psi_\tau(m), f \circ \psi_\tau^{-1} \rangle_M = f( \psi_\tau^{-1} (\psi_\tau(m))) =  f(m) = \langle m, f \rangle_M.
\]
Then proceed as in the proof of Lemma~\ref{taulifting}, defining an image of $\phi_\tau$ using the isomorphism $\psi$.

In conclusion, we have shown:

\begin{lem} \label{lem:taupluslifting}
 In case 3, there exists an $\O[\Gamma_{L,\tau \oplus \tau^*}]$-module $\widetilde{\tau \oplus \tau^*}$ with pairing lifting $\tau \oplus \tau^*$ together with its pairing.
\end{lem}

\subsection{Lifts with Pairings} \label{sec:liftingwithpairings}
We continue the notation of \S\ref{sec:decomposing}, and analyze how the duality pairing interacts with the decomposition \eqref{eq:mdecomposition}.  Recall that we obtained an isomorphism $M \simeq M^\vee$ of $R[\Gamma_L]$-modules which gave isomorphisms $M_\tau \simeq M_{\tau^*}^\vee$ of $R[\Gamma_{L,\tau}]$-modules.  The key point is that for any lift and extension $\tau'$ of $\tau$, the isomorphism of $R[\Lambda_L]$-modules
\[
 \tau' \otimes \Hom_{\Lambda_L}(\tau',M) \to M_\tau
\]
is compatible with the $\Gamma_{L,\tau}$-action.

To do this, it is convenient to break into the cases introduced at the end of \S\ref{sec:decomposing}.  For an irreducible $k[\Lambda]$-module $\tau$ occurring in $V$, note that $(\tau^g)^\vee = (\tau^\vee)^g$ for any $g \in \Gamma_L$, so if $\tau \simeq \tau^*$ then $\tau^g \simeq (\tau^g)^*$.  We let
\begin{itemize}
 \item $\Sigma_n$ denote the set of $\Gamma_L$-conjugacy classes of such $\tau$ for which $\tau$ is not conjugate to $\tau^*$;
 \item $\Sigma_e$ denote the set of $\Gamma_L$-conjugacy classes of such $\tau$ for which $\tau \simeq \tau^*$;
 \item $\Sigma_c$ denote the set of $\Gamma_L$-conjugacy classes of such $\tau$ for which $\tau^*$ is conjugate to $\tau$ but $\tau \not \simeq \tau^*$.
\end{itemize}
From~\eqref{eq:mdecomposition}, we obtain a decomposition
\begin{equation}
 M = \bigoplus_{\tau \in \Sigma_n} \Ind_{\Gamma_{L,\tau}}^{\Gamma_L} \left(\tau'  \otimes W_\tau  \right) \oplus \bigoplus_{\tau \in \Sigma_e} \Ind_{\Gamma_{L,\tau}}^{\Gamma_L}\left(\tau'  \otimes W_\tau \right) \oplus \bigoplus_{\tau \in \Sigma_c} \Ind_{\Gamma_{L,\tau}}^{\Gamma_L}\left(\tau' \otimes W_\tau \right)
\end{equation}
where $\tau'$ is any lift and extension of $\tau$ to $\Gamma_\tau$ and $W_\tau = \Hom_{\Lambda}(\tau',M)$ is a representation of $T_{L,\tau}$.  Note that $W_\tau$ is free as an $R$-module (since $M$ and $\tau'$ are, with $\tau' \neq 0$ and $R$ local), and hence that $W_\tau$ is tamely ramified of the type considered in \S\ref{sec:mrtame}.  

We may rewrite this to make use of the special extensions constructed in \S\ref{sec:extendingtau}.  In particular, for $\tau \in \Sigma_c$ we rewrite 
\[
 \Ind_{\Gamma_{L,\tau}}^{\Gamma_L}\left(\tau'  \otimes W_\tau \right) = \Ind_{\Gamma_{L,\tau \oplus \tau^*}}^{\Gamma_L} \left( \widetilde{\tau \oplus \tau^*} \otimes W_{\tau \oplus \tau^*} \right).                                                                                                                                                                                                                                                                                                                                                                                                                                                                                                                                                                                                  \]
This uses the notation and results from Case 3 in \S\ref{sec:extendingtau}, in particular the fact that $\tau \oplus \tau^*$ is an irreducible representation of the group $\Lambda_L'$ generated by $\Lambda_L$ and a $g \in \Gamma_L$ with $\tau^* \simeq \tau^g$, and the definition $W_{\tau \oplus \tau^*} : = \Hom_{\Lambda'_L}(\widetilde{\tau \oplus \tau^*},M)$.   Note that $W_{\tau \oplus \tau^*}$ is a representation of $T_{L,\tau \oplus \tau^*}$, which is a subgroup of $T_L = \Gamma_L / \Lambda_L$, hence of the form $T_q$ as considered in \S\ref{sec:mrtame}.  Using the extensions $\tildetau$ and $\widetilde{\tau \oplus \tau^*}$ from Cases $1$ and $2$ from \S\ref{sec:extendingtau}, we obtain a decomposition
\begin{align} \label{eq:modifieddecomposition}
 M = \bigoplus_{\tau \in \Sigma_n} \Ind_{\Gamma_{L,\tau}}^{\Gamma_L} \left(\tildetau  \otimes W_\tau  \right) \oplus \bigoplus_{\tau \in \Sigma_e} \Ind_{\Gamma_{L,\tau}}^{\Gamma_L}\left(\tildetau  \otimes W_\tau \right) \oplus \bigoplus_{\tau \in \Sigma_c} \Ind_{\Gamma_{L,\tau \oplus \tau^*}}^{\Gamma_L} \left( \widetilde{\tau \oplus \tau^*} \otimes W_{\tau \oplus \tau^*} \right)  .                              
\end{align}

Now let $M'$ be another $R[\Gamma_L]$-module that is finite free over $R$ such that the irreducible representations of $\Lambda_L$ occurring in $V':= M'/\m M'$ are among the same $\tau$'s, so
\[
 M' = \bigoplus_{\tau \in \Sigma_n} \Ind_{\Gamma_{L,\tau}}^{\Gamma_L} \left(\tildetau  \otimes W'_\tau  \right) \oplus \bigoplus_{\tau \in \Sigma_e} \Ind_{\Gamma_{L,\tau}}^{\Gamma_L}\left(\tildetau  \otimes W'_\tau \right) \oplus \bigoplus_{\tau \in \Sigma_c} \Ind_{\Gamma_{L,\tau \oplus \tau^*}}^{\Gamma_L} \left( \widetilde{\tau \oplus \tau^*} \otimes W'_{\tau \oplus \tau^*} \right).
\]

\begin{lem} \label{lem:schur}
The natural map
\[
\resizebox{\textwidth}{!} {$\displaystyle
 \bigoplus_{\tau \in \Sigma_n} \Hom_{T_{L,\tau}}(W_\tau,W'_\tau ) \oplus \bigoplus_{\tau \in \Sigma_e} \Hom_{T_{L,\tau}}(W_\tau,W'_\tau ) \oplus \bigoplus_{\tau \in \Sigma_c} \Hom_{T_{L,\tau \oplus \tau^*}}(W_{\tau \oplus \tau^*},W'_{\tau \oplus \tau^*} )   \to \Hom_{\Gamma_L}(M,M')$}
\]
is an isomorphism.
\end{lem}

\begin{proof}
We may immediately pass to working with representations of the finite discrete groups $\Gamma$ and $\Lambda$.  Notice that
\[
  \Hom_{\Gamma}(\Ind^{\Gamma}_{\Gamma_\tau} (M_\tau), \Ind^{\Gamma}_{\Gamma_\tau} ( M'_\tau)) \simeq \Hom_{\Gamma_\tau}( \Ind^{\Gamma}_{\Gamma_\tau} (M_\tau), M'_\tau) \simeq \Hom_{\Gamma_\tau}(M_\tau, M'_\tau)
\]
where the second isomorphism uses that $\Hom_{\Gamma_\tau}(M_{\tau^g}, M'_{\tau})=0$ by Lemma~\ref{lem:zero} when $\tau$ and $\tau^g$ are non-isomorphic.  Furthermore, if $\tau_1$ and $\tau_2$ are not $\Gamma$-conjugate then
\[
 \Hom_{\Gamma}(\Ind^{\Gamma}_{\Gamma_{\tau_1}} (M_{\tau_1}), \Ind^{\Gamma}_{\Gamma_{\tau_2}} ( M'_{\tau_2})) \simeq \Hom_{\Gamma_\tau}( \Ind^{\Gamma}_{\Gamma_{\tau_1}} (M_{\tau_1}) ,M_{\tau_2} )  =0
\]
using Lemma~\ref{lem:zero} as $\tau_1^g$ is not isomorphic to $\tau_2$ for any $g \in \Gamma$.  Then using \eqref{eq:mdecomposition} we see that
\[
 \Hom_{\Gamma}(M,M') = \bigoplus_{[\tau_1],[\tau_2]} \Hom_{\Gamma}(\Ind^{\Gamma}_{\Gamma_{\tau_1}} (M_{\tau_1}), \Ind^{\Gamma}_{\Gamma_{\tau_2}} ( M'_{\tau_2})) = \bigoplus_{[\tau]} \Hom_{\Gamma_\tau}(M_\tau, M'_\tau).
\]

All the irreducible finite-dimensional representations of $\Lambda$ occurring in $V$ and $V'$ are absolutely irreducible over $k$ by design.  For $\tau \in \Sigma_n \cup \Sigma_e$, consider the natural inclusion 
\begin{equation} \label{eq:wtauiso}
 \Hom_R(W_\tau,W'_\tau) \into \Hom_{\Lambda}(\tildetau \otimes W_\tau, \tildetau \otimes W'_\tau) = \Hom_{\Lambda}(\tildetau,\tildetau) \tensor{R} \Hom_R(W_\tau,W'_\tau),
\end{equation}
using that $W_\tau$ and $W'_\tau$ are $R$-free of finite rank and $\Lambda$ acts trivially.  But $R \into \Hom_{\Lambda}(\tildetau,\tildetau)$ is an isomorphism because $\End_{\Lambda}(\tau) = k$ and because surjectivity can be checked modulo $\m_R$ using Lemma~\ref{lem:hom}.  As $M_\tau \simeq \tildetau \otimes W_\tau$, this implies that
\begin{align*}
 \Hom_{\Gamma_\tau}(M_\tau,M'_\tau) = \Hom_\Lambda(M_\tau,M'_\tau)^{T_\tau} &= \Hom_\Lambda(\tildetau \otimes W_\tau, \tildetau \otimes W'_\tau) ^{T_\tau} \\
 &= \Hom_R(W_\tau,W_\tau)^{T_\tau} = \Hom_{T_\tau}(W_\tau,W'_\tau)
 \end{align*}
 where $T_\tau$ is the image of $T_{L,\tau}$ in $\Gamma_\tau$.  An analogous computation in the case $\tau \in \Sigma_c$ completes the proof.
\end{proof} 

We can now consider the duality isomorphism $M \simeq M^\vee$.  By Lemma~\ref{lem:schur}, this is equivalent to a collection of isomorphisms of $R[T_{L,\tau}]$-modules $\varphi_\tau: W_\tau \simeq W_{\tau^*}^\vee$ for $\tau \in \Sigma_e \cup \Sigma_n$ and an isomorphism of $R[T_{L,\tau \oplus \tau^*}]$-modules $\varphi_{\tau} : W_{\tau \oplus \tau^*} \simeq W_{\tau \oplus \tau^*}^\vee$ for $\tau \in \Sigma_c$.  We analyze the cases separately.

In {\bf Case 1} (when $\tau$ is not conjugate to $\tau^*$), note that $\Ind^{\Gamma_L}_{\Gamma_{L,\tau}} M_\tau$ is an isotropic subspace of $M$.  In particular, the perfect sign-symmetric pairing on $\Ind^{\Gamma_L}_{\Gamma_{L,\tau}} M_\tau \oplus \Ind^{\Gamma_L}_{\Gamma_{L,\tau^*}}  M_{\tau^*}$ is equivalent to an isomorphism of $R[\Gamma_{L}]$-modules $$\Ind^{\Gamma_L}_{\Gamma_{L,\tau}} M_\tau \simeq \left(\Ind^{\Gamma_L}_{\Gamma_{L,\tau^*}}  M_{\tau^*} \right)^\vee,$$ which is equivalent to the isomorphism of $R[T_{L,\tau}]$-modules $\varphi_\tau : W_\tau \simeq W_{\tau^*}^\vee$.  (Note that the similitude character $\nu$ is present in the use of the dual.) 

In {\bf Case 2} (when $\tau$ is isomorphic to $\tau^*$), the perfect sign-symmetric pairing on $\Ind^{\Gamma_L}_{\Gamma_{L,\tau}} M_\tau$ is equivalent to an isomorphism $W_\tau \simeq W_\tau^\vee$ of $R[T_{L,\tau}]$-modules.  Thus it gives a pairing $\langle , \rangle_{W_\tau}$ on $W_\tau$ via
\[
 \langle w_1, w_2 \rangle_{W_\tau} := \varphi_\tau(w_1)(w_2).
\]
We claim this pairing is sign-symmetric. 

From \S\ref{sec:extendingtau} we have an isomorphism $\imath: \tildetau \simeq \tildetau^\vee$ of $R[\Gamma_{L,\tau}]$-modules.
As at the end of \S\ref{sec:decomposing}, let $\psi : M \to M^\vee$ be the isomorphism of $R[\Gamma_L]$-modules given by $m \mapsto \langle m , - \rangle_M$, and define $\langle v_1,v_2 \rangle_{\tildetau} := \imath(v_1)(v_2)$.  We have a commutative diagram
\[
\xymatrixcolsep{5pc}
 \xymatrix{
  \tildetau \otimes W_\tau \ar^{\id \otimes \varphi_\tau} [r] \ar[d] &  \tildetau \otimes W_\tau^\vee \ar^{\imath \otimes \id} [r] & \tildetau^\vee \otimes W_\tau^\vee \ar[d] \\
  M_\tau \ar^{\psi} [rr]&  & M_\tau^\vee
 }
\]
The commutativity says that for elementary tensors
$m_i = v_i \otimes w_i \in M_\tau = \tildetau \otimes W_\tau$ we have
\begin{equation} \label{eq:pairingseq}
\begin{split}
 \langle m_1,m_2 \rangle_M = \psi(m_1)(m_2)& = \left( \imath(v_1) \otimes \varphi_\tau(w_1) \right) ( v_2 \otimes w_2) \\
  &= \imath(v_1)(v_2) \cdot \varphi_\tau(w_1)(w_2) = \langle v_1, v_2 \rangle_{\tau} \langle w_1, w_2 \rangle_{W_\tau}.
\end{split}
\end{equation}
The pairings are perfect and $\langle \cdot, \cdot \rangle_\tau$ is $\epsilon_\tau$-symmetric, so the pairing on $W_\tau$ is $\epsilon_{W_\tau}$-symmetric if and only if the pairing on $M_\tau$ is $\epsilon$-symmetric.  We have that $\epsilon = \epsilon_{W_\tau} \epsilon_{\tau}$.  

In {\bf Case 3} ($\tau \in \Sigma_c$), an analogous argument using the isomorphism $\widetilde{\tau \oplus \tau^*} \simeq \widetilde{\tau \oplus \tau^*}^\vee$ of $R[\Gamma_{L,\tau \oplus \tau^*}]$-modules (which define the $\epsilon_{W_{\tau \oplus \tau^*}}$-symmetric pairing on $\widetilde{\tau \oplus \tau^*}$) shows that the pairing induced by $\varphi_{\tau} : W_{\tau \oplus \tau^*} \simeq W_{\tau \oplus \tau^*}^\vee$ is $\epsilon_{\tau \oplus \tau^*}$-symmetric if and only if the pairing on 
\[
 \Ind_{\Gamma_{L,\tau \oplus \tau^*}}^{\Gamma_L} \left( \widetilde{\tau \oplus \tau^*} \otimes W'_{\tau \oplus \tau^*} \right)
\]
induced from the pairing on $M$ is sign-symmetric with sign $\epsilon = \epsilon_{\tau \oplus \tau^*} \epsilon_{W_{\tau \oplus \tau^*}}$.

\subsection{Minimally Ramified Deformations} \label{sec:mrcondition}
We can now define the minimally ramified deformation condition for $\rhobar : \Gamma_L \to G(k)$, under the continuing assumption that we have extended $k$ so all irreducible representations of $\Lambda_L$ occurring in $V$ are absolutely irreducible over $k$.  From~\eqref{eq:modifieddecomposition}, we obtain a decomposition
\begin{equation}
 V =  \bigoplus_{\tau \in \Sigma_n} \Ind_{\Gamma_{L,\tau}}^{\Gamma_L} \left(\tildetau  \otimes \overline{W}_\tau  \right) \oplus \bigoplus_{\tau \in \Sigma_e} \Ind_{\Gamma_{L,\tau}}^{\Gamma_L}\left(\tildetau  \otimes \overline{W}_\tau \right) \oplus \bigoplus_{\tau \in \Sigma_c} \Ind_{\Gamma_{L,\tau \oplus \tau^*}}^{\Gamma_L} \left( \widetilde{\tau \oplus \tau^*} \otimes \overline{W}_{\tau \oplus \tau^*} \right) .
\end{equation}
where $\overline{W}_\tau$ is a representation of $T_{L,\tau}$ over $k$ and $\overline{W}_{\tau \oplus \tau^*}$ is a representation of $T_{L,\tau \oplus \tau^*}$.

If $\tau \in \Sigma_n$, define $\overline{G}_\tau := \underline{\Aut}(\overline{W}_\tau)$.  If $\tau \in \Sigma_e$, there is a sign-symmetric perfect pairing $\langle \cdot , \cdot \rangle_{\overline{W}_\tau}$ on $\overline{W}_\tau$: in that case define $\overline{G}_\tau := \underline{\GAut}(\overline{W}_\tau,\langle \cdot, \cdot \rangle_{\overline{W}_\tau})$.  (The notation $\underline{\GAut}$ means automorphisms preserving the pairing up to scalar.)  If $\tau \in \Sigma_c$, there is a sign-symmetric perfect pairing on $\overline{W}_{\tau \oplus \tau^*}$: in that case define $\overline{G}_{\tau} := \underline{\GAut}(\overline{W}_{\tau\oplus \tau^*},\langle \cdot, \cdot \rangle_{\overline{W}_{\tau \oplus \tau^*}})$.
Make a finite extension of $k$ so that all the pairings are split.   Lift $\overline{G}_\tau$ to a split reductive group $G_\tau$ over $\O$ by lifting the split linear algebra data.

\begin{defn} \label{defn:minimallyramified}
Let $\rho : \Gamma_L \to G(R)$ be a continuous Galois representation lifting $\rhobar$ as above, with associated $R[\Gamma]$-module 
\[
 M = \bigoplus_{\tau \in \Sigma_n} \Ind_{\Gamma_{L,\tau}}^{\Gamma_L} \left(\tildetau  \otimes W_\tau  \right) \oplus \bigoplus_{\tau \in \Sigma_e} \Ind_{\Gamma_{L,\tau}}^{\Gamma_L}\left(\tildetau  \otimes W_\tau \right) \oplus \bigoplus_{\tau \in \Sigma_c} \Ind_{\Gamma_{L,\tau \oplus \tau^*}}^{\Gamma_L} \left( \widetilde{\tau \oplus \tau^*} \otimes W_{\tau \oplus \tau^*} \right) .
\]
We say that $\rho$ is \emph{minimally ramified} with similitude character $\nu$ if each $W_\tau$ and $W_{\tau \oplus \tau^*}$ is minimally ramified in the sense of Definition~\ref{defn:minimallyramifiedtame} as a representation of $T_{L,\tau}$ or $T_{L,\tau \oplus \tau^*}$ valued in the group $G_\tau$ with specified similitude character.  (Note that defining the minimally ramified deformation condition as in \S\ref{sec:mrtame} may require an additional \'{e}tale local extension of $\O$, which as always is harmless for applications.)
\end{defn}

Let $D^{\mr, \nu}_{\rhobar}$ denote the deformation functor for $\rhobar$ with specified similitude character $\nu$, and $\D^{\mr}_{G_\tau}$ (respectively $\D^{\mr,\nu}_{G_\tau}$) denote the deformation functor for $\overline{W}_\tau$ or $\overline{W}_{\tau \oplus \tau^*}$ viewed as a representation valued in $G_\tau$ (respectively with specified similitude character $\nu$).  In particular, letting $r = \dim \overline{W}_\tau$ (or $\dim \overline{W}_{\tau \oplus \tau^*}$ when $\tau \in \Sigma_c$), we have that the adjoint representation $\on{ad} \overline{W}_\tau$ is the Lie algebra of $\overline{G}_\tau$, which is the Lie algebra of $\GSp_r$ or $\GO_r$ when $\tau \in \Sigma_e$ or $\Sigma_c$, and the Lie algebra of $\GL_r$ when $\tau \in \Sigma_n$.  Let $\Sigma'_n$ consist of one representative for each pair of representations $\tau, \tau^* \in \Sigma_n$.
 
\begin{prop} \label{prop:defproduct}
There is a natural isomorphism of functors
\[
 D^{\mr\, \nu}_{\rhobar} \to \prod_{\tau \in \Sigma'_n } D^{\mr}_{G_\tau}  \times \prod_{\tau \in \Sigma_e} D^{\mr \, \nu}_{G_\tau}\times \prod_{\tau \in \Sigma_c} D^{\mr \, \nu}_{G_\tau} .
\]
\end{prop}

\begin{proof}
This expresses the decomposition obtained in this section: given a lift $\rho$ of $\rhobar$, we obtain a decomposition of $M$ as in Definition~\ref{defn:minimallyramified}.  Our analysis with pairings shows that when $\tau \in \Sigma_e$, $W_\tau$ is a deformation of $\overline{W}_\tau$ together with its $\epsilon_{W_\tau}$-symmetric perfect pairing.  Likewise, when $\tau \in \Sigma_c$ we know that $W_{\tau \oplus \tau^*}$ is a deformation of $\overline{W}_{\tau \oplus \tau^*}$ together with its $\epsilon_{W_{\tau \oplus \tau^*}}$-symmetric pairing.  When $\tau \in \Sigma_n$, we know $W_\tau \simeq W_{\tau^*}^\vee$.  This gives the natural map: to $\rho \in  D^{\mr\, \nu}_{\rhobar}(R)$ associate the collection of the $W_\tau$ for  $\tau \in \Sigma_e \cup \Sigma_c \cup \Sigma'_n$.

Conversely, given $W_\tau$ for $\tau \in \Sigma_e \cup \Sigma_c \cup \Sigma'_n$, and defining $W_{\tau^*} := W_\tau^\vee$ for $\tau \in \Sigma'_n$ we can define a lift
\[
 M :=  \bigoplus_{\tau \in \Sigma_n} \Ind_{\Gamma_{L,\tau}}^{\Gamma_L} \left(\tildetau  \otimes W_\tau  \right) \oplus \bigoplus_{\tau \in \Sigma_e} \Ind_{\Gamma_{L,\tau}}^{\Gamma_L}\left(\tildetau  \otimes W_\tau \right) \oplus \bigoplus_{\tau \in \Sigma_c} \Ind_{\Gamma_{L,\tau \oplus \tau^*}}^{\Gamma_L} \left( \widetilde{\tau \oplus \tau^*} \otimes W_{\tau \oplus \tau^*} \right) .
\]
as in \eqref{eq:mdecomposition}.  (Note that the groups $\Gamma_{L,\tau}$ depend only on the fixed residual representation $V$.)  For $\tau \in \Sigma_e$, the perfect pairing on the lift $W_\tau$ gives an isomorphism $\varphi_\tau : W_\tau \simeq W_\tau^\vee$ of $R[T_{L,\tau}]$-modules, which  gives a sign-symmetric pairing (with sign $\epsilon_{W_\tau} \epsilon_\tau = \epsilon$) on $\Ind^{\Gamma_L}_{\Gamma_{L,\tau}} (\tildetau \otimes W_\tau)$ (using equation \eqref{eq:pairingseq}).  Likewise, for $\tau \in \Sigma_c$ the sign-symmetric pairing on $W_{\tau \oplus \tau^*}$ gives one on $\Ind_{\Gamma_{L,\tau \oplus \tau^*}}^{\Gamma_L} \left( \widetilde{\tau \oplus \tau^*} \otimes W_{\tau \oplus \tau^*} \right)$.  For $\tau \in \Sigma_n$, we obtain an isomorphism $\varphi_\tau : W_\tau \simeq W_{\tau^*}^\vee$ of $R[T_{L,\tau}]$-modules and hence an $\epsilon$-symmetric perfect pairing on $(\tildetau \otimes W_\tau) \oplus (\tildetau^\vee \otimes W_{\tau^*})$ which gives one on $\Ind^{\Gamma_L}_{\Gamma_{L,\tau}} \left( \tildetau \otimes W_{\tau} \right) \oplus \Ind^{\Gamma_L}_{\Gamma_{L,\tau}} \left( \tildetau^\vee \otimes W_{\tau^*} \right)$.
Putting these together, we obtain a sign-symmetric pairing on $M$; the action of $\Gamma_L$ preserves it up to scalar, giving a continuous homomorphism $\rho : \Gamma_L \to G(R)$.

Finally, we claim that these constructions are compatible with strict equivalence of lifts, giving an identification of the deformation functors.  For $g \in \Ghat(R)$, decompose the $g$-conjugate $\Gamma_L$-representation $M^g$ according to \eqref{eq:modifieddecomposition}.  As $g$ reduces to the identity, it must respect the decomposition into $\tau$-isotypic pieces, so gives automorphisms $g_\tau \in \Aut(W_\tau)$ and $g_\tau \in \Aut(W_{\tau \oplus \tau^*})$.  If $\tau \in \Sigma_e$ or $\Sigma_c$, as $g$ is compatible with the pairing on $M$ we see $g_\tau$ is compatible with the pairing as well.  For $\tau \in \Sigma_e$, the $g_\tau$-conjugate $T_{L,\tau}$-representation $W_\tau^{g_\tau}$ is minimally ramified as minimally ramified lifts of $\overline{W}_\tau$ for the group $T_{L,\tau}$ are a deformation condition, and likewise for $\tau \in \Sigma_c$ and $\tau \in \Sigma'_n$.

Conversely, given $g_\tau \in \Aut(W_\tau)$ reducing to the identity (compatible with the pairing on $W_\tau$ or $W_{\tau \oplus \tau^*}$ if there is one), using \eqref{eq:mdecomposition} and acting on each piece we obtain a lift of the form $M^g$ for  $g \in \Ghat(R)$.  Thus the identification is compatible with strict equivalence.
\end{proof}

\begin{cor} \label{cor:minimallyramified}
Under the assumptions \ref{assumption1}-\ref{assumption4} needed to analyze the tame case, and our standing assumption that all the irreducible representations of $\Lambda_L$ appearing in $V$ are absolutely irreducible over $k$, the minimally ramified deformation condition with fixed similitude character is liftable.  The dimension of the tangent space is $h^0(\Gamma_L,\adzerorho)$.
\end{cor}

\begin{proof}
Liftability is a consequence of Proposition~\ref{prop:defproduct} and the smoothness of the minimally ramified lifting ring for representations of $T_{L,\tau}$ (Proposition~\ref{prop:smoothnessmrtame} and Corollary~\ref{cor:mrtamefixed}). 
By Corollary~\ref{cor:mrtamefixed}, for $\tau \in \Sigma_e$
the dimension of the tangent space of $D^{\mr,\nu}_{G_\tau}$ is $h^0(T_{L,\tau},\on{ad} \overline{W}_\tau) - 1 = h^0(T_{L,\tau}, \on{ad}^0 \overline{W}_\tau)$, and for $\tau \in \Sigma_c$ the dimension is $h^0(T_{L,\tau \oplus \tau^*},\on{ad} \overline{W}_{\tau \oplus \tau^*}) - 1 = h^0(T_{L,\tau\oplus \tau^*}, \on{ad}^0 \overline{W}_{\tau \oplus \tau^*})$.  For $\tau \in \Sigma'_n$, by Proposition~\ref{prop:smoothnessmrtame}  the dimension of the tangent space of $D^{\mr}_{G_\tau}$ is $h^0(T_{L,\tau},\on{ad} \overline{W}_\tau)$.  Using Proposition~\ref{prop:defproduct}, we see that the dimension of the tangent space of the minimally ramified deformation condition is
\[
 \sum_{\tau \in \Sigma_e} h^0(T_{L,\tau},\on{ad}^0 \overline{W}_\tau) + \sum_{\tau \in \Sigma_c} h^0(T_{L,\tau\oplus \tau^*},\on{ad}^0 \overline{W}_{\tau \oplus \tau^*}) +\sum_{\tau \in \Sigma'_n} h^0(T_{L,\tau},\on{ad} \overline{W}_\tau).
\]

It remains to identify this quantity with $h^0(\Gamma_L,\adzerorho)$.  Using Lemma~\ref{lem:schur}
\begin{align*}
H^0(\Gamma_L,\End({V})) &= \Hom_{k[\Gamma_L]}({V},{V}) \\
&= \bigoplus_{\tau \in \Sigma_e \cup \Sigma_n} \Hom_{T_{L,\tau}}(\overline{W}_\tau ,\overline{W}_\tau)  \oplus \bigoplus_{\tau \in \Sigma_c} \Hom_{T_{L,\tau \oplus \tau^*}}(\overline{W}_{\tau \oplus \tau^*},\overline{W}_{\tau \oplus \tau^*}) \\
&= \bigoplus_{\tau \in \Sigma_e \cup \Sigma_n} H^0 (T_{L,\tau}, \End(\overline{W}_\tau))  \oplus \bigoplus_{\tau \in \Sigma_c} H^0 (T_{L,\tau \oplus \tau^*}, \End(\overline{W}_{\tau \oplus \tau^*})) .
\end{align*}
We are interested in $H^0(\Gamma_L, \adzerorho)$: the elements $\psi \in H^0(\Gamma_L,\End({V}))$ compatible with the pairing on ${V}$ in the sense that for $v, v' \in {V}$
\[
 \langle \psi v, \psi v' \rangle = \langle v, v' \rangle.
\]
The pairing on ${V}_\tau = \tau \otimes \overline{W}_\tau$ is induced by the pairings on $\overline{W}_\tau$ and $\tau$ when $\tau \in \Sigma_e$, and is induced by the pairings on $\overline{W}_{\tau \oplus \tau^*}$ and $\tau \oplus \tau^*$ when $\tau \in \Sigma_c$.  When $\tau \in \Sigma'_n$, the pairing on ${V}_\tau \oplus {V}_{\tau^*}$ comes from the $\Gamma_{L,\tau}$-isomorphism $V_\tau \simeq V_{\tau^*}^\vee$ which in turn comes from the $T_{L,\tau}$-isomorphism $\overline{W}_\tau \simeq \overline{W}_{\tau^*}^\vee$.  So $\psi$ is compatible with the pairing if and only if 
\begin{itemize}
 \item when $\tau \in \Sigma_e$, the associated $\psi_\tau \in H^0 (T_{L,\tau}, \End(\overline{W}_\tau))$ is compatible with the pairing on $\overline{W}_\tau$;
 \item  when $\tau \in \Sigma_c$, the associated $\psi_\tau \in H^0(T_{L,\tau\oplus \tau^*}, \End(\overline{W}_{\tau \oplus \tau^*})$ is compatible with the pairing on $\overline{W}_{\tau \oplus \tau^*}$;
 \item when $\tau \in \Sigma'_n$, the associated $\psi_\tau$ and $\psi_{\tau^*}$ are identified by duality and the isomorphism $\overline{W}_\tau \simeq \overline{W}_{\tau^*}^\vee$. 
\end{itemize}
In the first two cases, $\ad^0 \overline{W}_\tau$ and $\ad^0 \overline{W}_{\tau \oplus \tau^*}$ are the symplectic or orthogonal Lie algebra, consisting exactly of endomorphisms compatible with the pairing on $\overline{W}_\tau$.  In the third, we just choose one of $\psi_\tau$ and $\psi_{\tau^*}$ without restriction, which determines the other.  Thus we see
\[
 H^0(\Gamma_L,\adzerorho) = \bigoplus_{\tau \in \Sigma_e} H^0(T_{L,\tau}, \on{ad}^0 \overline{W} _\tau) \oplus \bigoplus_{\tau \in \Sigma_c} H^0(T_{L,\tau \oplus \tau^*}, \on{ad}^0 \overline{W} _{\tau \oplus \tau^*}) \oplus \bigoplus_{ \tau \in \Sigma'_n} H^0(T_{L,\tau}, \on{ad} \overline{W} _\tau). 
\]
\end{proof}

The Corollary is a more precise version of Theorem~\ref{thm:maintheorem}.


\def\cftil#1{\ifmmode\setbox7\hbox{$\accent"5E#1$}\else
  \setbox7\hbox{\accent"5E#1}\penalty 10000\relax\fi\raise 1\ht7
  \hbox{\lower1.15ex\hbox to 1\wd7{\hss\accent"7E\hss}}\penalty 10000
  \hskip-1\wd7\penalty 10000\box7} \def\cprime{$'$} \def\cprime{$'$}
\providecommand{\bysame}{\leavevmode\hbox to3em{\hrulefill}\thinspace}
\providecommand{\MR}{\relax\ifhmode\unskip\space\fi MR }
\providecommand{\MRhref}[2]{%
  \href{http://www.ams.org/mathscinet-getitem?mr=#1}{#2}
}
\providecommand{\href}[2]{#2}

\end{document}